\documentclass[reqno,12pt]{amsart}
\usepackage[margin=1.2in]{geometry}
\usepackage{amssymb}
\usepackage{amsthm}
\usepackage{amsmath}
\usepackage{amsxtra}
\usepackage{latexsym}
\usepackage{mathrsfs}
\usepackage[all,cmtip]{xy}
\usepackage[all]{xy}
\usepackage{enumitem}
\usepackage{xcolor}
\usepackage{comment}
\usepackage{mathabx,epsfig}
\usepackage{bm}
\usepackage{mathtools}
\usepackage{booktabs}
\usepackage{appendix}
\usepackage{hyperref}
\hypersetup{
    colorlinks,
    citecolor=black,
    filecolor=black,
    linkcolor=black,
    urlcolor=black
}

\usepackage[
backend=biber,
bibstyle=ieee-alphabetic,
citestyle=ieee-alphabetic,
sorting=nyt,
isbn=false,
url=false,
doi=false,
giveninits=true,
maxnames=10,
labelalpha=true,
maxalphanames=4,
dashed=false,
]{biblatex}
\DeclareCaseLangs{}

\DeclareFieldFormat*{date}{\mkbibparens{#1}}
\DeclareFieldFormat*{volume}{\mkbibbold{#1}}
\DeclareFieldFormat*{title}{\mkbibemph{#1\isdot}}
\DeclareFieldFormat*{journaltitle}{#1}
\DeclareFieldFormat*{pages}{{#1}}

\AtEveryBibitem{%
  \ifentrytype{online}{%
    \DeclareFieldFormat*{date}{#1}%
  }{%
  }%
}

\newtheorem{thm}{Theorem}[section]
\newtheorem{lem}[thm]{Lemma}
\newtheorem{conj}[thm]{Conjecture}
\newtheorem{claim}[thm]{Claim}
\newtheorem{prop}[thm]{Proposition}
\newtheorem{cor}[thm]{Corollary}

\theoremstyle{definition}
\newtheorem{conv}[thm]{Convention}

\newtheorem{eg}[thm]{Example}
\newtheorem{defn}[thm]{Definition}
\newtheorem{rmk}[thm]{Remark}
\newtheorem*{ack}{Acknowledgements}
\numberwithin{equation}{section}
\newcommand{\var}{\overline}

\theoremstyle{plain}

\newcommand{\pullbackcorner}[1][dr]{\save*!/#1-1.5pc/#1:(-1,1)@^{|-}\restore}


\def\Q{{\mathbb Q}}

\def\Z{{\mathbb Z}}
\def\P{{\mathbb P}}

\def\F{{\mathbb F}}
\def\m{{\mathfrak m}}
\def\n{{\mathfrak n}}

\def\var{\overline}
\def\Hom{\mathop{\mathrm{Hom}}\nolimits}

\def\sO{\mathcal{O}}

\DeclareMathOperator{\Aut}{Aut}

\DeclareMathOperator{\Spec}{Spec}
\DeclareMathOperator{\Proj}{Proj}

\DeclareMathOperator{\Ker}{Ker}

\DeclareMathOperator{\codim}{codim}

\DeclareMathOperator{\sht}{ht}

\DeclareMathOperator{\chara}{char}
\DeclareMathOperator{\jac}{jac}
\DeclareMathOperator{\Fitt}{Fitt}
\DeclareMathOperator{\length}{length}

\AtBeginDocument{%
  \def\MR#1{}
}

\newenvironment{claimproof}[0]
  {%
   \paragraph{\it Proof.}%
  }
  {%
    \hfill$\blacksquare$%
  }

\makeatletter
\def\@tocline#1#2#3#4#5#6#7{\relax
  \ifnum #1>\c@tocdepth 
  \else
    \par \addpenalty\@secpenalty\addvspace{#2}%
    \begingroup \hyphenpenalty\@M
    \@ifempty{#4}{%
      \@tempdima\csname r@tocindent\number#1\endcsname\relax
    }{%
      \@tempdima#4\relax
    }%
    \parindent\z@ \leftskip#3\relax \advance\leftskip\@tempdima\relax
    \rightskip\@pnumwidth plus4em \parfillskip-\@pnumwidth
    #5\leavevmode\hskip-\@tempdima
      \ifcase #1
       \or\or \hskip 1em \or \hskip 2em \else \hskip 3em \fi%
      #6\nobreak\relax
    \hfill\hbox to\@pnumwidth{\@tocpagenum{#7}}\par
    \nobreak
    \endgroup
  \fi}
\makeatother

\title[Fedder type criteria for quasi-$F$-splitting II]
{Fedder type criteria for quasi-$F$-splitting II}
\author{Tatsuro Kawakami, Teppei Takamatsu, and Shou Yoshikawa}
\address{Graduate School of Mathematical Sciences, University of Tokyo, 3-8-1 Komaba,
Meguro-ku, Tokyo 153-8914, Japan}
\email{kawakami@ms.u-tokyo.ac.jp}

\address{Department of Mathematics, Graduate School of Science, Kyoto University, Kyoto 606-8502, Japan}
\email{teppeitakamatsu.math@gmail.com}

\address{Institute of Science Tokyo, Tokyo 152-8551, Japan}
\email{yoshikawa.s.9fe9@m.isct.ac.jp}
\begin{document}

\begin{abstract}
In this paper, we apply Fedder-type criteria for quasi-$F$-splitting to provide explicit computations of quasi-$F$-split heights for Calabi-Yau hypersurfaces, bielliptic surfaces, Fano varieties, and rational double points.
We also find interesting phenomena concerned with inversion of adjunction, fiber products, Fano varieties, and general fibers of fibrations.
\end{abstract}

\maketitle

\setcounter{tocdepth}{2}

\tableofcontents

\section{Introduction}

The notions of \emph{quasi-$F$-splitting} and \emph{quasi-$F$-split height} extend and quantify the classical concept of $F$-splitting. 
For an $\mathbb{F}_p$-scheme $X$, having quasi-$F$-split height equal to one is equivalent to being $F$-split, and $X$ is said to be \emph{quasi-$F$-split} if its quasi-$F$-split height is finite. 
The concept of quasi-$F$-splitting was originally introduced by Yobuko \cite{Yobuko19} and has since been studied in greater depth in \cite{KTTWYY, KTTWYY2, KTTWYY3, KTY1, TWY}, among others. 
This paper is a continuation of \cite{KTY1}.

For $F$-splitting, there exists a well-known and powerful criterion---\emph{Fedder’s criterion}---which enables one to determine whether a hypersurface (or more generally a complete intersection) is $F$-split. 
This enables us to compute the $F$-splitting of hypersurfaces. 
In \cite{KTY1}, we generalized Fedder’s criterion and established analogous \emph{Fedder-type criteria} for determining the quasi-$F$-split height of complete intersections. 
Depending on the geometric or algebraic conditions on $X$, several variants of these criteria were obtained (see \emph{loc.\ cit.} for details).

In this paper, we investigate several interesting consequences of applying these Fedder-type criteria for quasi-$F$-splitting. 
Specifically, we establish the following results.

\begin{description}
\item[Inversion of adjunction] \textup{}\\
We construct an example of a non-quasi-$F$-split hypersurface $X$ whose hyperplane section $Y$ is quasi-$F$-split (Example~\ref{eg:counterexample to inversion of adjunction}). 
Conversely, if $X \subseteq \mathbb{P}^N$ is a hypersurface of degree $d$ and $Y$ is the intersection of $N+1-d$ general hyperplanes, then $\sht(Y) \ge \sht(X)$ (Proposition~\ref{cor:hypersurface inversion of adjunction}). 
The same holds for hypersurfaces in weighted projective spaces, providing useful upper bounds for the quasi-$F$-split heights of certain Fano varieties (cf.~Proposition \ref{prop:sm del Pezzo}). 
For related variants of inversion of adjunction, see~\cite[Section~4]{KTTWYY}, \cite[Section~5]{TWY}, and \cite[Section~5.4]{Kawakami-Tanaka2}.

\item[Fiber products] \textup{}\\
We show that if $X$ and $Y$ are complete intersections and $X \times Y$ is not quasi-$F$-split, then at least one of $X$ or $Y$ must be $F$-split (Proposition~\ref{prop:fiber product not quasi-F-split}). 
Conversely, if $X$ is $F$-split, then $\sht(X \times Y)=\sht(Y)$ (Proposition~\ref{prop:fiber products}). 
Here $X$ and $Y$ need not to be complete intersections. 
Our proof uses a Fedder-type criterion for more general rings~\cite[Theorem~4.8]{KTY1}. 
After the appearance of the first version of this paper, Yobuko~\cite[Theorem~1.3]{YobukoHodgeWitt} independently proved similar results without the complete-intersection assumption, using a completely different method.

\item[General fibers] \textup{}\\
We show that quasi-$F$-splitting is inherited by a general fiber when the generic fiber has trivial canonical divisor and is a complete intersection in a projective space over the function field of the base (Corollary~\ref{cor:fiber space}). 
As a consequence, we deduce that quasi-$F$-split surfaces do not admit quasi-elliptic fibrations (Corollary~\ref{thm:no quasi-elliptic fibration}). 
In contrast, when the canonical divisor is non-trivial, this inheritance fails. 
For example, let $X$ be the smooth Fano threefold
\[
\{\,x_0y_0^2 + x_1y_1^2 + x_2y_2^2 = 0\,\}
\subset \mathbb{P}^2_{[x_0:x_1:x_2]} \times \mathbb{P}^2_{[y_0:y_1:y_2]}
\]
over a perfect field of characteristic $p=2$. 
A Fedder-type criterion shows that $X$ is quasi-$F$-split of height~2. 
However, the projection $X \to \mathbb{P}^2_{[x_0:x_1:x_2]}$ defines a \emph{wild conic bundle}, whose fibers are all non-reduced and hence not quasi-$F$-split.

\item[Calabi--Yau hypersurfaces of arbitrary Artin--Mazur height] \textup{}\\
As an application of a Fedder-type criterion, we prove the existence of a Calabi--Yau hypersurface over $\overline{\mathbb{F}}_p$ whose Artin--Mazur height $h$ can be any positive integer $h \in \mathbb{Z}_{>0}$ and for any prime $p$ (Theorem \ref{thm:unbounded}). 
Note that for Calabi--Yau varieties, the Artin--Mazur height coincides with the quasi-$F$-split height \cite{Yobuko19}.

\item[Bielliptic surfaces] \textup{}\\
Smooth projective surfaces of Kodaira dimension $0$ are classified into four types: abelian, bielliptic, K3, and Enriques surfaces. 
For all but the bielliptic case, Yobuko~(\cite{Yobuko19}, \cite{Yobuko2}, \cite{YobukoHodgeWitt}) completely determined the quasi-$F$-split height. 
We determine $\sht(X)$ in the remaining case, when $X$ is (quasi-)bielliptic. 
More precisely, we show that $\sht(X) = \sht(\widetilde{X})$ for the abelian surface cover $\widetilde{X}$ of $X$, except for a few exceptional cases. 
In those exceptional cases---including quasi-bielliptic surfaces---we prove that $X$ is \emph{not} quasi-$F$-split (Theorem~\ref{thm:bielliptic}).
This gives the first example of a finite \'{e}tale morphism along which
quasi-$F$-splitting does not descend.
Moreover, this example arises from a morphism between varieties with trivial canonical bundles (for which, $F$-splitting automatically descends).

\item[Fano varieties] \textup{}\\
We construct a non-quasi-$F$-split smooth Fano $d$-fold for every $d>2$. 
On the other hand, we prove that every smooth del Pezzo surface (that is, a smooth Fano surface) has quasi-$F$-split height at most two and hence is quasi-$F$-split (Proposition~\ref{prop:sm del Pezzo}; see also~\cite[Theorem~F]{KTTWYY} for another proof). 
Note that non-$F$-split del Pezzo surfaces exist when $p \le 5$~\cite[Example~5.5]{Hara}. 
For quasi-$F$-splitting of log del Pezzo surfaces, see \cite[Theorem~B]{KTTWYY} and \cite[Theorem~D]{KTTWYY2}. 
For quasi-$F$-splitting of higher-dimensional Fano varieties, we refer to \cite{Kawakami-Tanaka2, Kawakami-Tanaka4}, where it is shown that every smooth Fano threefold of Picard rank $>1$ or Fano index $>1$ is quasi-$F$-split.

\item[Rational double points] \textup{}\\
Rational double points (RDPs) in characteristic $p$ are $F$-split when $p>5$, but may fail to be $F$-split when $p \le 5$~\cite{Hara2}. 
By combining a Fedder-type criterion for quasi-$F$-splitting with Artin’s classification of RDPs in low characteristics~\cite{Artin}, we determine the quasi-$F$-split heights of all RDPs and, in particular, show that all RDPs are quasi-$F$-split (Theorem~\ref{rdp}). 
For instance, we prove that an RDP defined by
\[
\{ z^2 + x^2y + xy^n = 0 \} \quad (n \ge 2)
\]
in characteristic $p=2$ has quasi-$F$-split height $\lceil \log_2 n \rceil$, which is unbounded as $n$ increases. 
Yobuko had previously computed quasi-$F$-split heights of RDPs (except those of type $D_n$ in $p=2$) in unpublished work using detailed local-cohomology calculations---a completely different approach from ours. 
In general, every two-dimensional klt singularity is known to be quasi-$F$-split~\cite[Theorem~C]{KTTWYY}, but this relies on methods involving the Cartier operator and does not yield explicit height computations.
\end{description}

\begin{ack}
The authors wish to express our gratitude to Hiromu Tanaka, Fuetaro Yobuko, and Jakub Witaszek for valuable discussion.
They are also grateful to Shunsuke Takagi, Kenta Sato, Masaru Nagaoka, Naoki Imai, Tetsushi Ito, Yuya Matsumoto, Ippei Nagamachi, and Yukiyoshi Nakkajima for helpful comments.
The first author was supported by JSPS KAKENHI Grant number JP19J21085.
The second author was supported by JSPS KAKENHI Grant number JP19J22795, JP22KJ1780, and JP25K17228. 
The third author was supported by JSPS KAKENHI Grant number JP24K16889.
\end{ack}

\section{Preliminaries}

\subsection{Notation and terminologies}
Let $S$ be a ring of characteristic $p>0$.
\begin{itemize}
    \item $F \colon S \to S$ denotes the absolute Frobenius homomorphism.
    If $F$ is finite, then $S$ is said to be \emph{$F$-finite}.
    \item For an $S$-module $M$, we define a new $S$-module structure on the same underlying abelian group by
    \[
    a \cdot m := a^pm
    \]
    for $a \in S$ and $m \in M$.
    This module is denoted by $F_*M$.
    To distinguish the $S$-action on $F_*M$ from that on $M$, we write elements of $F_*M$ as $F_*m$.
    Thus,
    \[
    aF_*m = F_*(a^pm)
    \]
    for $a \in S$ and $m \in M$.
    With this notation, the Frobenius map
    \[
    S \to F_*S,\quad a \mapsto F_*(a^p) = aF_*1
    \]
    is an $S$-module homomorphism.
    We similarly define the iterated versions $F^n_*M$ and $F^n_*m$.
    \item Let $I$ be an ideal of $S$.
    We define $I^{[p]}$ to be the ideal generated by
    $\{a^p \mid a \in I\}$.
    Furthermore, we inductively set $I^{[p^n]}:=(I^{[p^{n-1}]})^{[p]}$ for $n \in \Z_{>0}$.
    Note that $I \cdot F^n_*S = F^n_*I^{[p^n]}$.
    \item A scheme $X$ is called a {\em variety} if it is an integral scheme that is separated and of finite type over a field $k$.
    \item Let $A$ be a ring.
    Since the polynomial ring $A[x_1,\ldots,x_N]$ is a free $A$-module with basis given by monomials, every element
    \[
    f \in A[x_1,\ldots,x_N]
    \]
    admits a unique expression
    \[
    f=\sum a_i M_i,
    \]
    where $a_i \in A$ and $M_i$ are distinct monic monomials.
    This expression is called the \emph{monomial decomposition} of $f$ over $A$.
\end{itemize}

Following \cite{KTY1}, we use the following notation.
\begin{itemize}
    \item 
    For any $n \in \Z_{>0}$ and an $\F_p$-algebra $A$, $W_{n} (A)$ denotes the ring of Witt vectors of length $n$.
    We denote the quotient ring $W_n (A)/ (p)$ by $\var{W}_n (A)$.
    We denote the Frobenius, Verschiebung, and restriction morphisms by
    \begin{eqnarray*}
    &F& \colon W_n (A) \rightarrow W_n (A), \\
    &V& \colon W_n (A) \rightarrow W_{n+1} (A),\\
    &R& \colon W_{n+1} (A) \rightarrow W_n (A),
    \end{eqnarray*}
    respectively. For $a \in A$, $[a]\in W(A)$ denotes the Teichm\"uller lift of $a$. 
    For an ideal $I \subset A$, we put
    \[
    W_n (I) := \ker (W_n (A) \rightarrow W_n (A/I)).
    \]
    \item 
     For any $n \in \Z_{>0}$ and an $\F_p$-scheme $X$, $W_n \sO_X$ denotes the sheaf of rings of Witt vectors of length $n$, $\var{W}_n \sO_X$ denotes the quotient ring sheaf $W_{n} \sO_X/(p)$,
     and $W_n X$ denotes the scheme $(X, W_{n} \sO_X)$.
     For any morphism of $\F_p$-schemes $f \colon X\rightarrow Y$, $W_n f$ denotes the induced morphism $W_n X \rightarrow W_n Y$.
     We also define $F$, $V$, and $R$ similarly as for $W_n (A)$.
\end{itemize}

\begin{lem}\label{lem:etale base change for Witt Frobenius}
Let $f \colon Y \to X$ be an \'etale morphism of $F$-finite $\F_p$-schemes.
Then we have 
\[
(W_{n}f)^*F_*W_n\sO_X \cong F_*W_n\sO_Y.
\]
\end{lem}

\begin{proof}
Localizing at a point of $Y$, we may assume $Y=\Spec B$ and $X=\Spec A$. 
We prove that the natural ring homomorphism
\[
\phi \colon F_*W_n(A) \otimes_{W_n(A)} W_n(B) \to F_*W_n(B)
\]
is an isomorphism.
By \cite[Propositions A.8 and A.5]{LZ}, both sides are free $F_*W_n(A)$-modules of the same rank.
Therefore it suffices to show that $\phi$ is surjective.
Take $b \in B$ and $1 \leq r \leq n$.
Since $F^r_*A \otimes_{A} B \cong F^r_*B$, there exist $b' \in B$ and $a \in A$ such that $b=b'^{p^r}a$.
Thus,
\[
\phi(V^{r-1} ([a]) \otimes [b'])=V^{r-1} ([a] )\,[b'^p]=V^{r-1} ([ab'^{p^r}] )=V^{r-1}[b].
\]
Hence $\phi$ is surjective.
\end{proof}

\begin{prop}\label{prop:completion for witt ring}
Let $(S,\m)$ be an $F$-finite Noetherian local ring of positive characteristic and $\widehat{S}$ its completion.
Then $W_n(S)$ is a Noetherian local ring.
Let $\n$ be the maximal ideal of $W_n(S)$.
Then the $\n$-adic completion of $W_n(S)$ is isomorphic to $W_n(\widehat{S})$.
\end{prop}

\begin{proof}
By \cite[Proposition A.4]{LZ}, $W_n(S)$ is Noetherian.
Let $I$ be the kernel of $R^{n-1} \colon W_n(S) \to S$.
Then $I=\mathrm{Im}(V)$.
Since $(\mathrm{Im}V)^i \subseteq \mathrm{Im}(V^i)$ by the ring structure of $W_n(S)$ and $\mathrm{Im}(V^n)=0$, it follows that $I$ is nilpotent.
As a nilpotent ideal is contained in every maximal ideal, we conclude that $W_n(S)$ is local.

Next, consider the diagram
\[
\xymatrix{
W_n(S) \ar[r] \ar[d] \ar@{}[rd]|{\circlearrowright} & W_n(\widehat{S}) \ar[d] \\
W_n(S/\m^k) \ar[r] & W_n(S/\m^k).
}
\]
We note that $W_n(\widehat{S})$ coincides with the inverse limit of the system $\{W_n(S/\m^k)\}$, since the limit can be computed in the category of sets.
Therefore, $W_n(\widehat{S})$ is the completion of $W_n(S)$ with respect to $\{W_n(\m^k)\}_k$.
Thus it remains to show that $\{W_n(\m^k)\}$ and $\{\n^k\}$ define the same topology.

Fix $k>0$.
Since $W_n(S/\m^k)$ is Artinian by \cite[Proposition A.4]{LZ}, the image of $\n^l$ is zero for sufficiently large $l$, hence $\n^l \subseteq W_n(\m^k)$.
On the other hand, there exists $l>0$ such that $\m^l \subseteq (\m^k)^{[p^n]}$.
We claim that $W_n(\m^l) \subseteq \n^k$.
Since $W_n(\m^l)$ is generated by $V^{r-1}[\m^l]$ for $r \leq n$, it suffices to show that $V^{r-1}[(\m^k)^{[p^n]}] \subseteq \n^k$ for $r \leq n$.
Now, $[\m^k] \subseteq \n^k$ and $\mathrm{Im}(V) \subseteq \n$.
Hence, for $a \in \m^k$ we have
\[
V^{r-1}([a^{p^n}])=[a^{p^{n-r}}]\cdot V[1] \in \n^k.
\]
Therefore, $W_n(\m^l) \subseteq \n^k$, as desired.
\end{proof}

\subsection{Quasi-\texorpdfstring{$F$}{F}-splitting}
In this subsection, we collect some basic properties of quasi-$F$-splitting.

\begin{defn}[\cite{Yobuko2}]
\label{defn:quasi-F-split}
Let $X$ be an $F$-finite $\F_p$-scheme.
For a positive integer $n$, we say that $X$ is \textit{$n$-quasi-$F$-split} if 
there exists a $W_n\sO_X$-homomorphism $\phi \colon F_*W_n\sO_X \to \sO_X$ making the following diagram commute:
\[
\xymatrix{
W_n\sO_X \ar[r]^-{F} \ar[d]_-{R^{n-1}} & F_*W_n\sO_X \ar@{-->}[ld]^{\exists \phi} \\
\sO_X.
}
\]

We define the \emph{quasi-$F$-split height} $\sht(X)$ as the infimum of positive integers $n$ such that $X$ is $n$-quasi-$F$-split.
If no such $n$ exists, we set $\sht(X)=\infty$.
We say that $X$ is \emph{quasi-$F$-split} if $\sht(X)<\infty$.
If $\sht(X)=1$, then we say that $X$ is \emph{$F$-split}.
Note that $X$ is $F$-split if and only if the Frobenius map
$\sO_X \to F_*\sO_X$
splits as an $\sO_X$-module homomorphism.
\end{defn}






\begin{rmk}\label{rmk:evaluation}
Let $X$ be an $F$-finite $\F_p$-scheme and $n$ a positive integer.
Then $X$ is $n$-quasi-$F$-split if and only if the evaluation map
\[
\Hom_{W_nX}(F_*W_n\sO_X,\sO_X) \to H^0(X,\sO_X)
\]
is surjective.
\end{rmk}

\begin{prop}\label{prop:fiber space}
Let $\pi \colon Y \to X$ be a morphism of $F$-finite  $\F_p$-schemes.
Suppose that $\sO_X \to \pi_*\sO_Y$ splits as an $\sO_X$-module homomorphism. 
If $Y$ is $n$-quasi-$F$-split, then so is $X$.
\end{prop}

\begin{proof}
Suppose that $Y$ is $n$-quasi-$F$-split.
Then there exists $\phi \colon F_*W_n\sO_Y \to \sO_Y$ fitting into the commutative diagram of Definition \ref{defn:quasi-F-split}.
Applying $(W_n \pi)_*$, we obtain
\[
\xymatrix{
W_n(\pi_*\sO_Y) \ar[r]^-F \ar[d]_-{R^{n-1}} & F_*W_n(\pi_*\sO_Y) \ar[ld]^-{W_n\pi_*\phi} \\
\pi_*\sO_Y,
}
\]
where we used $(W_{n}\pi)_*W_n\sO_Y \cong W_n(\pi_*\sO_Y)$.
Thus we obtain a composition
\[
F_*W_n\sO_X \to F_*W_n\pi_*\sO_Y \overset{W_n\pi_*\phi}{\longrightarrow} \pi_*\sO_Y \to \sO_X,
\]
which maps $1$ to $1$, where the last map is the given splitting of $\sO_X \to \pi_*\sO_Y$.
Therefore, $X$ is $n$-quasi-$F$-split.
\end{proof}

\begin{rmk}
Proposition \ref{prop:fiber space} applies, for instance, to an algebraic fiber space, to a finite morphism of degree prime to $p$ between normal integral schemes, and to a base change by a field extension.
\end{rmk}

\begin{prop}\label{prop:small birational case}
Let $X$ (resp.~$Y$) be a Noetherian $F$-finite $\F_p$-scheme satisfying Serre's condition $(S_2)$
and $U \subseteq X$ (resp.~$V \subseteq Y$) an open subscheme satisfying
$\codim_X(X \backslash U)\geq 2$ (resp.~$\codim_Y(Y \backslash V) \geq 2$).
Suppose that there exists an isomorphism $U\cong V$.
Then we have $\sht(X)=\sht(Y)$.
\end{prop}

\begin{proof}
It suffices to show that if $X$ is $n$-quasi-$F$-split, then so is $Y$.
Since $X$ is $n$-quasi-$F$-split, the same holds for $U$, and in particular for $V$.
As $Y$ satisfies $(S_2)$, we have $\iota_*\sO_V = \sO_Y$.
By Proposition \ref{prop:fiber space}, it follows that $Y$ is $n$-quasi-$F$-split, as desired.
\end{proof}


\begin{prop}\label{prop:etale extension}
Let $f \colon Y \to X$ be a morphism of $F$-finite $\F_p$-schemes.
Suppose that $X$ is $n$-quasi-$F$-split.
\begin{itemize}
    \item[\textup{(1)}] If $f$ is \'etale, then $Y$ is $n$-quasi-$F$-split.
    \item[\textup{(2)}] If $X$ and $Y$ are Noetherian schemes satisfying Serre's condition $(S_2)$ and $f$ is \'etale in codimension one, then $Y$ is $n$-quasi-$F$-split.
\end{itemize}
\end{prop}

\begin{proof}
Since $X$ is $n$-quasi-$F$-split, there exists a homomorphism $\phi \colon F_*W_n\sO_X \to \sO_X$ fitting into the commutative diagram
\[
\xymatrix{
W_n\sO_X \ar[r]^-F \ar[d]_{R^{n-1}}  & F_*W_n\sO_X \ar[ld]^-\phi \\
\sO_X. &
}
\]

We first prove (1). 
Assume that $f$ is \'etale.
Pulling back the above diagram via $W_{n}f$ and applying Lemma \ref{lem:etale base change for Witt Frobenius}, we obtain a commutative diagram
\[
\xymatrix{
W_n\sO_Y \ar[r]^-F \ar[d]_{W_nf^*R^{n-1}}  & F_*W_n\sO_Y \ar[ld]^-{W_nf^*\phi} \\
W_nf^*\sO_X. &
}
\]
Composing $W_nf^*\phi$ with the natural map $W_nf^*\sO_X \to W_n\sO_Y$, we conclude that $Y$ is $n$-quasi-$F$-split.

Next, we prove (2).
Let $V \subseteq Y$ be the \'etale locus of $f$ and set $U:=f(V)$.
Since $\mathrm{codim}_Y(Y \setminus V) \geq 2$ by assumption, it follows from Proposition \ref{prop:small birational case} that $\sht(Y)=\sht(V)$.
By (1), $V$ is $n$-quasi-$F$-split, and hence so is $Y$.
\end{proof}

\begin{cor}\label{cor:base change}
Let $k \subseteq k'$ be an algebraic separable extension of $F$-finite fields of positive characteristic.
Let $X$ be an $F$-finite $k$-scheme and set $X':=X \times_k k'$.
Then $\sht(X)=\sht(X')$. 
\end{cor}

\begin{proof}
Since $k \subseteq k'$ is a split inclusion, we have $\sht(X') \geq \sht(X)$.
As $k \subseteq k'$ is an inductive limit of \'etale extensions, it follows that $f \colon X' \to X$ is a projective limit of \'etale morphisms.
By Lemma \ref{lem:etale base change for Witt Frobenius}, we obtain
\[
(W_nf)^*F_*W_n\sO_X \cong F_*W_n\sO_{X'}.
\]
Therefore, the argument of Proposition \ref{prop:etale extension} shows that $\sht(X')=\sht(X)$.
\end{proof}

\begin{prop}\label{prop:completion}
Let $(R,\m)$ be an $F$-finite Noetherian local ring of positive characteristic and let $\widehat{R}$ be the $\m$-adic completion of $R$.
Then $\sht(R)=\sht(\widehat{R})$.
\end{prop}

\begin{proof}
By \cite[Proposition A.4]{LZ}, $W_n(R)$ is Noetherian.
By Proposition \ref{prop:completion for witt ring}, the map $W_n(R) \to W_n(\widehat{R})$ is the $\n$-adic completion, where $\n$ is the maximal ideal of $W_n(R)$.
Since $F \colon W_n(R)\to F_*W_n(R)$ is finite by \cite[Lemma~2.4]{KTY1}, it follows that 
\[
F_* (W_n(R)) \otimes_{W_n(R)} W_n(\widehat{R})
\]
is the $\n$-adic completion of $F_*W_n(R)$. 
Similarly, we have
\[
R \otimes_{W_n(R)} W_n(\widehat{R}) \cong \widehat{R}.
\]
Using Remark \ref{rmk:evaluation} together with the faithful flatness of $W_n(R) \to W_n(\widehat{R})$, we obtain $\sht(R)=\sht(\widehat{R})$.
\end{proof}

\subsection{Fedder type criterion for quasi-$F$-splitting}

In this subsection, we recall the results in \cite{KTY1}.

\begin{conv}\label{conv:example}
Let $k$ be a perfect field of characteristic $p>0$, 
$S \coloneqq k[x_1,\ldots,x_N]$ the polynomial ring, 
$\m \coloneqq (x_1,\ldots,x_N)$, and $R \coloneqq S_\m$.
We set $\m^{[p]} \coloneqq (x_1^p,\ldots,x_N^p) \subset S$.

Consider the $S$-basis 
\[
\bigl\{\,F_*(x_1^{i_1}\cdots x_N^{i_N}) \,\bigm|\, 0 \le i_1,\ldots,i_N \le p-1 \,\bigr\}
\]
of $F_*S$, and let $u \in \Hom_S(F_*S,S)$ be the $S$-linear map which sends 
$F_*((x_1\cdots x_N)^{p-1})$ to $1$ and all other basis vectors to $0$.
Then $u$ is a generator of $\Hom_S(F_*S,S)$.

We define a map
\[
\Delta_1 \colon S \longrightarrow S
\]
as follows.  
For $a \in S$, write $a=\sum_{j=1}^r a_j M_j$ with $a_j \in k$ and each $M_j$ a monomial of coefficient $1$.
Define
\[
\Delta_1(a)
:= \sum_{\substack{0 \le \alpha_1,\ldots,\alpha_r \le p-1\\ \alpha_1+\cdots+\alpha_r=p}}
\frac{1}{p}\binom{p}{\alpha_1,\ldots,\alpha_r}  (a_1M_1)^{\alpha_1}\cdots (a_rM_r)^{\alpha_r}.
\]

Let $f'_1,\ldots,f'_m \in \m$ be a regular sequence and set $f \coloneqq f'_1\cdots f'_m$.
We define an $S$-linear map
\begin{align*}
\theta \colon F_*S &\longrightarrow S,\\
F_*a &\longmapsto u\bigl(F_*(\Delta_1(f^{p-1})a)\bigr).
\end{align*}
We define a sequence of ideals $\{I_n\}_{n\ge1}$ as follows:
\begin{align*}
I_1 
&\coloneqq (f^{p-1}) + ((f'_1)^p,\ldots,(f'_m)^p),\\
I_{n+1} 
&\coloneqq \theta\Bigl(F_*I_n \cap \Ker(u)\Bigr)
     + I_1.
\end{align*}
\end{conv}

\begin{thm}[{{\cite[Theorem~4.11]{KTY1}}}]\label{thm:Fedder's criterion}
With the notation as in Convention~\ref{conv:example}, we have
\[
\sht\bigl((S/(f'_1,\ldots,f'_m))_\m\bigr)
=\inf\bigl\{n \in \Z_{\ge 1} \bigm| I_n \nsubseteq \m^{[p]} \bigr\}.
\]
\end{thm}

\begin{thm}[{\cite[Theorem~5.8]{KTY1}}]
\label{CY-fedder}
With the notation of Convention~\ref{conv:example}, we endow \(S\) with a \(\mathbb{Z}_{\ge 0}^m\)-graded structure such that each \(x_i\) is homogeneous, and we assume that \(f'_1, \ldots, f'_m\) are homogeneous and that 
\(\deg(f) = \deg(f'_1) + \cdots + \deg(f'_m) = N.\)
We define $f_n$ for $n\in\Z_{>0}$ by $f_1:=f^{p-1}$ and
\[
f_n:=f^{p-1}\Delta_1(f^{p-1})^{p^{n-2}+\cdots+1}
\]
for $n \geq 2.$
Then
\[
\sht(S/(f'_1,\ldots,f'_m))=\mathrm{inf}\{n \mid \theta^{n-1}(F^{n-1}_* (f^{p-1}k )) \nsubseteq \m^{[p]} \}=\mathrm{inf}\{n \mid f_n \notin \m^{[p^n]} \}
\]
holds.
\end{thm}

\begin{cor}[{\cite[Corollary~4.17]{KTY1}}]\label{cor:Fedder's criterion for quasi-F-splitting}
We use the notation in Convention \ref{conv:example}.
Then the following hold.
\begin{itemize}
    \item[\textup{(1)}] There exists the smallest ideal $I_{\infty}$ satisfying
    \[
    I_{\infty} \supseteq \theta(F_{*}I_{\infty} \cap \Ker(u) )+(I^{[p]} \colon I).
    \]
    \item[\textup{(2)}] $(S/(f'_1,\ldots,f'_m))_\m$ is quasi-$F$-split if and only if $I_{\infty} \nsubseteq \m^{[p]}$.
\end{itemize}
\end{cor}

\begin{cor}[{\cite[Corollary~4.19]{KTY1}}]\label{cor:non-q-split criterion}
We use the notation in Convention \ref{conv:example}.
Then the following hold.
\begin{itemize}
    \item[\textup{(1)}] If $f^{p-2} \in \m^{[p]}$, then $\sht((S/(f'_1,\ldots,f'_m))_\m)=\infty$.
    \item[\textup{(2)}] If $ (f^{p-2},I^{[p]})f^{p(p-2)}\Delta_{1}(f) \subseteq \m^{[p^2]}$ 
    and $f^{p-1} \in \m^{[p]}$, then $\sht((S/(f'_1,\ldots,f'_m))_\m)=\infty$. 
\end{itemize}
\end{cor}

\section{Quasi-$F$-splitting for weighted multigraded rings}

\subsection{Quasi-\texorpdfstring{$F$}{F}-split height for weighted multigraded rings}\label{Appendix:multigrade}
In this subsection, we define section rings and coordinate rings for closed subvarieties of fiber products of weighted projective spaces.  
Furthermore, we compare the quasi-$F$-splitting of a projective variety with that of its section ring, which enables us to apply a Fedder-type criterion for quasi-$F$-splitting to complete intersections in fiber products of weighted projective spaces.

\begin{defn}\label{wdefn:well-formed}
Let $P=\P(a_1:\cdots:a_N)$ be a weighted projective space over a field $k$.
We say that $P$ is \emph{well-formed} if for every $1 \leq i \leq N$, we have
\[
\gcd(a_j \mid j \neq i)=1.
\]
Every weighted projective space is isomorphic to a well-formed weighted projective space, by an argument similar to \cite[Proposition~1.3]{Dol}.
For an integer $n$, we define $\sO_P(n)$ as in \cite[Section~1.4]{Dol}.
\end{defn}

\begin{prop}\label{wprop:first properties}
Let $P:=\P(a_1:\cdots:a_N)$ be a well-formed weighted projective space over a field $k$.
Let $U_i:=D_{+}(x_i)$ and $V_i:=\bigcap_{i \neq j} U_j$.
Let $S:=k[x_1,\ldots,x_N]$ be a polynomial ring graded by $\deg(x_i)=a_i$.
Then, for every $n,m \in \Z$,
\begin{itemize}
    \item[\textup{(1)}] the natural map $S_n \to H^0(P,\sO_P(n))$ is bijective,
    \item[\textup{(2)}] $\sO_P(n)$ is Cohen--Macaulay,
    \item[\textup{(3)}] $\sO_{V_i}(n):=\sO_P(n)|_{V_i}$ is generated by $\prod_{i \neq j} x_j^{l_j}$, where $l_j$ are integers satisfying $\sum a_j l_j=n$,
    \item[\textup{(4)}] the reflexive hull of $\sO_P(n) \otimes \sO_P(m)$ is isomorphic to $\sO_P(n+m)$, and
    \item[\textup{(5)}] if $k$ is of positive characteristic, then the reflexive hull of $F^*\sO_P(n)$ is $\sO_P(pn)$.
\end{itemize}
\end{prop}

\begin{proof}
Assertion (1) follows from an argument similar to \cite[Theorem~1.4]{Dol}.
Let $f \in S$ be a homogeneous element of positive degree and set $U:=D_+(f)$.
Then
\[
\sO_P(n)(U)=S[f^{-1}]_n,
\]
which is a Cohen--Macaulay $S[f^{-1}]_0$-module since it is a direct summand of $S[f^{-1}]$.
Therefore $\sO_P(n)$ is  Cohen--Macaulay.

We now prove (3).
Without loss of generality, assume $i=N$.
Since $P$ is well-formed, $a_1,\ldots,a_{N-1}$ have no common factor.
Hence there exist integers $l_1,\ldots,l_{N-1}$ such that
\[
a_1l_1+\cdots+a_{N-1}l_{N-1}=n,
\]
so that the degree of $f:=\prod x_i^{l_i}$ is equal to $n$.
Thus $f$ is a section of $\sO_P(n)(V_i)=S[(x_1\cdots x_{N-1})^{-1}]_n$.
Next we show that $f$ is a generator.
Take $g \in \sO_P(n)(V_i)$.
Every monomial appearing in $g$ has degree $n$, so we may assume $g=\prod x_i^{m_i}$.
Since $g\in S[(x_1 \cdots x_{N-1})^{-1}]$, we have $m_N \geq 0$.
Define
\[
h:=(\prod_{i\leq N-1} x_i^{m_i-l_i}) x_N^{m_N}.
\]
As $\deg(g)=n$, the degree of $h$ is
\[
\sum_{i\leq N-1} a_i(l_i-m_i)+a_Nm_N=n-n=0.
\]
Furthermore, since $m_N\geq 0$, we have 
\[
h \in \sO_P(V_i)=S[(x_1\cdots x_{N-1})^{-1}]_0.
\]
By construction, $g=fh$, so $f$ is indeed a generator.

We now prove (4).
There is a natural map
\begin{align}\label{map}
    \sO_P(n)\otimes \sO_P(m) \to \sO_P(n+m).
\end{align}
Thus we obtain a map from the reflexive hull of $\sO_P(n)\otimes \sO_P(m)$ to $\sO_P(n+m)$ by (2).
We claim this map is an isomorphism.
Note that the open subset $V_1\cup \cdots \cup V_N$ contains all codimension-one points.
Indeed, its complement is
\[
\bigcap_{1 \leq i \leq N} \bigcup_{j \neq i} V_+(x_j) \subseteq \bigcup_{1 \leq i < j \leq N} (V_+(x_i)\cap V_+(x_j)),
\]
and since $\codim(V_+(x_i)\cap V_+(x_j))=2$, it follows that $V_1\cup\cdots\cup V_N$ contains all codimension-one points.
As $P$ is normal, it suffices to show that the map \eqref{map} is an isomorphism on each $V_i$, which follows from the description of generators in (3).

Finally, using the natural map $F^*\sO_P(n)\to \sO_P(pn)$, assertion (5) follows from the same argument as above.
\end{proof}

\begin{defn}\label{wdefn:notation for multiprojective space}
Let $P_1,\ldots,P_m$ be well-formed weighted projective spaces over a field $k$ of positive characteristic, and set $P:=P_1 \times \cdots \times P_m$.
We define the sheaves $\sO_P(h_1,\ldots,h_m)$ as the sheaves generated by homogeneous sections of degree $(h_1,\ldots,h_m)$, or equivalently as the reflexive hull of
\[
p_1^*\sO_{P_1}(h_1) \otimes \cdots \otimes p_m^*\sO_{P_m}(h_m),
\]
where $p_i \colon P \to P_i$ is the $i$-th projection.
By Proposition \ref{wprop:first properties}, we have
\[
(\sO_P(n) \otimes \sO_P(n'))^{**} \cong \sO_P(n+n'), \qquad (F^*\sO_P(n))^{**} \cong \sO_P(pn),
\]
where $(-)^{**}$ denotes the reflexive hull and $n,n' \in \Z^{m}$.

Let $j \colon X \hookrightarrow P$ be a closed immersion from a scheme $X$ satisfying condition $(S_2)$.
Assume that the regular locus $P_{\mathrm{reg}}$ of $P$ contains all codimension-one points of $X$.
For $h \in \Z^m$, we define $\sO_X(h)$ as the reflexive hull of $j^*\sO_P(h)$.

We define the \emph{coordinate ring} of $X$ as the image of the natural ring homomorphism
\[
\bigoplus_{h \in \Z^m} H^0(P,\sO_P(h)) \longrightarrow \bigoplus_{h \in \Z^m} H^0(X,\sO_X(h)),
\]
and the \emph{section ring} of $X$ as
\[
\bigoplus_{h \in \Z^m} H^0(X,\sO_X(h)).
\]
Both the section ring and the coordinate ring carry natural multigraded structures.
\end{defn}

\begin{lem}\label{wlem:projectively normality of multigraded projective space}
We use the notation of Definition \ref{wdefn:notation for multiprojective space}.
For each $i$, take a graded polynomial ring $R_i$ with $\Proj R_i \cong P_i$.
Define the multigraded polynomial ring $R$ by
\[
R:=R_1 \otimes \cdots \otimes R_m.
\]
Then the section ring of $P$ coincides with $R$.
\end{lem}

\begin{proof}
There exists a natural multigraded homomorphism
\[
\varphi \colon R \to \bigoplus_{h \in \Z^m} H^0(P,\sO_P(h)).
\]
Note that every sheaf $\sO_P(h)$ satisfies $(S_2)$ and $P$ is normal.
By K\"unneth's formula on the regular locus of $P$, we obtain
\[
\bigoplus_{h \in \Z^m} H^0(P,\sO_P(h)) \cong \bigoplus_{(h_1,\ldots,h_m) \in \Z^m} H^0(P_1,\sO_{P_1}(h_1)) \otimes \cdots \otimes H^0(P_m,\sO_{P_m}(h_m)).
\]
A standard argument shows that the natural homomorphism
\[
\varphi_i \colon R_i \to \bigoplus_{h_i \in \Z} H^0(P_i,\sO_{P_i}(h_i))
\]
is an isomorphism for each $i$.
Since $\varphi = \varphi_1 \otimes \cdots \otimes \varphi_m$, the claim follows.
\end{proof}

\begin{lem}\label{lem:multigraded structure on Witt ring}
Let $S$ be a $\mathbb{Z}_{\geq 0}^m$-graded ring of positive characteristic.
Then $W_n(S)$ admits a $\tfrac{1}{p^{\,n-1}}\mathbb{Z}_{\geq 0}^m$-graded structure defined by
\[
W_n(S)_{h/p^{\,n-1}} := (S_{h/p^{\,n-1}},\ldots,S_h),
\]
for $h \in \mathbb{Z}_{\geq 0}^m$, where we set $S_{h/p^k}=0$ whenever $h/p^k \notin \mathbb{Z}_{\geq 0}^m$.
\end{lem}

\begin{proof}
This follows from \cite[Proposition~7.1]{KTTWYY2}.
\end{proof}

\begin{defn}
We use the notation of Definition \ref{wdefn:notation for multiprojective space}.
Let $h \in \tfrac{1}{p^{\,n-1}}\mathbb{Z}^m$.
We define a $W_n\sO_P$-module $W_n\sO_P(h)$ as the sheaf generated by homogeneous sections of degree $h$. 
If $h \in \Z^m$, then $W_n\sO_P(h)$ is the Teichm\"uller lift of $\sO_P(h)$ in the sense of \cite[Definition~3.10]{tanaka22}.
For a closed immersion $W_nX \to W_nP$, we define $W_n\sO_X(h)$ as the reflexive hull of the pullback of $W_n\sO_P(h)$.
\end{defn}

\begin{lem}\label{wlem:first properties witt ver}
We use the notation of Definition \ref{wdefn:notation for multiprojective space}.
Let $U$ be the regular locus of $P$ and $U_X:=U \cap X$.
Then, for all $h \in \Z^m$,
\begin{itemize}
    \item[\textup{(1)}] $(F^*\sO_{U_X}(h)) \cong \sO_{U_X}(ph)$ 
    \item[\textup{(2)}] $\sO_{U_X} \otimes_{W_n\sO_{U_X}} W_n\sO_{U_X}(h) \cong \sO_{U_X}(h)$, and
    \item[\textup{(3)}] $F^*W_n\sO_{U_X}(h) \cong W_n\sO_{U_X}(ph)$.
\end{itemize}
\end{lem}

\begin{proof}
As in the proof of Proposition \ref{wprop:first properties}, we take an open affine covering $\{V_i\}$ of $U$ and a generator $f_i$ of $\sO_P(h)(V_i)$.
Then $\sO_X(h)(V_i \cap X)$ is also generated by $f_i$.
By the proof of Proposition \ref{wprop:first properties}, $f_i^p$ is a generator of $\sO_P(ph)(V_i)$, and thus $f_i^p$ is also a generator of $\sO_X(ph)(V_i)$.
Therefore, we obtain the assertion (1).

Next, we prove the assertion (2) and (3).
It is enough to show that the Teichm\"uller lift $[f_i]$ of $f_i$ is a generator of $W_n\sO_P(h)(V_i)$.
We know that $W_n\sO_P(h)(V_i)$ is generated by
\[
\{V^r[a] \mid a \in \sO_P(p^rh)(V_i) \}
\]
by the graded structure of the ring of Witt vectors.
We take an element $a$ in $\sO_P(p^rh)(V_i)$.
Since $f_i^{p^r}$ is a generator of $\sO_P(p^rh)(V_i)$, there exists $b \in \sO_P(V_i)$ such that $f_i^{p^r}b=a$.
Therefore, we have
\[
V^r[a]=V^r[f^{p^r}b]=[f_i]V^r[b].
\]
Since $V^r[b]$ is a section of $W_n\sO_P(V_i)$, we obtain the desired assertion.
\end{proof}

\begin{lem}\label{wlem:natural map}
We use the notation of Definition \ref{wdefn:notation for multiprojective space}.
Let $R$ be the section ring of $P$ and $S$ the coordinate ring of $X$.
There exists a natural multigraded ring homomorphism
\[
W_n(R) \longrightarrow \bigoplus_{h \in \tfrac{1}{p^{n-1}}\mathbb{Z}_{\geq 0}^m} H^0(P, W_n\sO_P(h)).
\]
Moreover, this map induces the multigraded  ring homomorphism
\[
W_n(S) \longrightarrow \bigoplus_{h \in \tfrac{1}{p^{n-1}}\mathbb{Z}_{\geq 0}^m} H^0(X, W_n\sO_X(h)).
\]
\end{lem}

\begin{proof}
Let $a \in R$ be a homogeneous element of degree $h \in \mathbb{Z}_{\geq 0}^m$.
Then $a$ is a section of $H^0(P,\sO_P(h))$, hence a homogeneous global section of degree $h$.
Consequently, $V^{r}[a]$ is a homogeneous global section of degree $h/p^r$, and we may regard it as an element of 
\[
H^0(P,W_n\sO_P(h/p^r)).
\]
Since $W_n(R)$ is generated by such elements $V^r[a]$, we obtain a natural ring homomorphism
\[
W_n(R) \to \bigoplus_{h \in \tfrac{1}{p^{\,n-1}}\mathbb{Z}_{\geq 0}^m} H^0(P,W_n\sO_P(h)).
\]
As $V^r[a]$ is homogeneous of degree $h/p^r$ in $W_n(R)$, this map is multigraded.

For the second assertion, note that there is a natural surjection $W_n(R) \twoheadrightarrow W_n(S)$, together with a natural map
\[
\bigoplus_{h \in \tfrac{1}{p^{n-1}}\mathbb{Z}_{\geq 0}^m} H^0(P, W_n\sO_P(h))
\longrightarrow 
\bigoplus_{h \in \tfrac{1}{p^{n-1}}\mathbb{Z}_{\geq 0}^m} H^0(X, W_n\sO_X(h)).
\]
Take a homogeneous $a \in R$ of degree $h$ lying in the kernel of $R \to S$.
Then $a$ is in the kernel of
\[
H^0(P,\sO_P(h)) \to H^0(X,\sO_X(h)).
\]
By the same argument as above, $V^r[a]$ can be regarded as a section of $H^0(P,\sO_P(h/p^r))$, and it lies in the kernel of
\[
H^0(P,\sO_P(h/p^r)) \to H^0(X,\sO_X(h/p^r)).
\]
Thus the desired induced map is obtained.
\end{proof}

\begin{prop}\label{wprop:multigraded algebra and proj}
We follow the notation of Definition \ref{wdefn:notation for multiprojective space}.
Let $S$ be the coordinate ring of $X$ and let $n$ be a positive integer.
Consider the following conditions:
\begin{itemize}
    \item[\textup{(1)}] $\sht(S) \leq n$,
    \item[\textup{(2)}] there exists a homomorphism $\theta \colon F_*W_n(S) \to S$ of multigraded $W_n(S)$-modules fitting into the commutative diagram
    \[
    \xymatrix{
    W_n(S) \ar[r]^F \ar[d]_{R^{n-1}} & F_*W_n(S) \ar[ld]^\theta \\
    S, &
    }
    \]
    \item[\textup{(3)}] $\sht(X) \leq n$.
\end{itemize}
Then $(1) \Rightarrow (2) \Rightarrow (3)$.
In particular, $\sht(X) \leq \sht(S)$.
Moreover, if $S$ coincides with the section ring of $X$, then $(3) \Rightarrow (1)$, and hence $\sht(S)=\sht(X)$.
\end{prop}

\begin{proof}
We first prove $(1)\Rightarrow(2)$.
Suppose there exists $\phi \colon F_*W_n(S) \to S$ as in Definition \ref{defn:quasi-F-split}.
Since $S$ is $F$-finite, \cite[Lemma~2.4]{KTY1} implies that $F_*W_n(S)$ is a finite multigraded $W_n(S)$-module.
Thus $\Hom_{W_n(S)}(F_*W_n(S),S)$ has a natural multigraded structure.
By decomposing $\phi$ into homogeneous components, we write $\phi=\sum \phi_h$ with
\[
\phi_h \colon F_*W_n(S) \to S(h/p^{n-1}).
\]
Since $\phi(1)=1$, it follows that $\phi_0(1)=1$.
Then $\phi_0$ satisfies the commutative diagram in (2).
As $\phi_0$ is a multigraded homomorphism, we obtain $(1)\Rightarrow(2)$.

Next we show $(2)\Rightarrow(3)$.
Take a homomorphism $\phi \colon F_*W_n(S) \to S$ as in (2).
We construct a $W_n\sO_X$-module homomorphism $\phi_X \colon F_*W_n\sO_X \to \sO_X$ sending $1$ to $1$.
Since $X$ is covered by open subsets
\[
\{D_+(a) \mid a \ \text{homogeneous of positive degree}\}, \qquad D_+(a)=\Spec S[1/a]_0,
\]
and $W_n\sO_X(D_+(a))=W_n(S[1/a]_0)$, it suffices to define $\phi_a$ locally.
For homogeneous $a,b\in S$ with $\deg(a)=\deg(b)$, we have
\[
V^{r-1}[b/a] \in W_n(S[1/a]_0),
\]
and by rewriting,
\[
V^{r-1}[b/a]=F ([1/a]) \cdot V^{r-1}[a^{p^r-1}b].
\]
Define
\[
\phi_a(V^{r-1}[b/a]) := \frac{1}{a}\phi(V^{r-1}[a^{p^r-1}b]).
\]
Since $\deg(\phi(V^{r-1}[a^{p^r-1}b]))=\deg(a)$, this lies in $S[1/a]_0$, and hence $\phi_a$ is well-defined.
Gluing the local maps $\phi_a$ yields $\phi_X$, and thus $X$ is $n$-quasi-$F$-split.

Finally, assume that $S$ coincides with the section ring and prove $(3)\Rightarrow(1)$.
Suppose there exists $W_n\sO_X$-module homomorphism $\phi_X \colon F_*W_n\sO_X \to \sO_X$ as in Definition \ref{defn:quasi-F-split}.
Then, for \(h \in \mathbb{Z}^m\), the map \(\phi_X\) induces a \(W_n\sO_X\)-module homomorphism
\[
F_*W_n\sO_X(ph) \longrightarrow \sO_X(h).
\]
Consequently, we obtain a 
\(\bigoplus H^0(X, W_n\sO_X(h))\)-module homomorphism
\[
\psi \colon 
F_* \!\left(
\bigoplus_{h \in \tfrac{1}{p^{n-1}}\mathbb{Z}^m} H^0(X, W_n\sO_X(h))
\right)
\longrightarrow 
\bigoplus_{h \in \mathbb{Z}^m} H^0(X, \sO_X(h))
\cong S.
\]
By Lemma~\ref{wlem:natural map}, there exists a natural ring homomorphism
\[
\varphi \colon 
W_n(S) \longrightarrow 
\bigoplus_{h \in \tfrac{1}{p^{n-1}}\mathbb{Z}^m} H^0(X, W_n\sO_X(h)).
\]
Composing these maps gives
\[
\psi \circ (F_*\varphi) \colon 
F_*W_n(S) \longrightarrow S,
\]
which sends \(1\) to \(1\).

Furthermore, this map is a \(W_n(S)\)-module homomorphism. Indeed, it follows from
\begin{align*}
    \psi(F_*(\varphi(F(\alpha)\beta))) 
    &\overset{(\star_1)}{=} \psi(F_*(\varphi(F(\alpha))\,\varphi(\beta))) \\
    &\overset{(\star_2)}{=} \psi(F_*(F(\varphi(\alpha)\varphi(\beta)))) \\
    &\overset{(\star_3)}{=} \varphi(\alpha)\,\psi(F_*\varphi(\beta)) \\
    &\overset{(\star_4)}{=} \alpha\,\psi(F_*\varphi(\beta)),
\end{align*}
where \((\star_1)\) and \((\star_2)\) follow from the fact that 
\(\varphi\) is a ring homomorphism commuting with the Frobenius endomorphism,
\((\star_3)\) follows because \(\psi\) is a 
\(\bigoplus H^0(X, W_n\sO_X(h))\)-module homomorphism, and
\((\star_4)\) follows from the ring isomorphism
\[
\bigoplus_{h \in \mathbb{Z}^m} H^0(X, \sO_X(h)) \cong S.
\]

This proves (1) and completes the proof.
\end{proof}

\begin{rmk}\label{rmk:projectively normality}
Let $k$ be a field.
\begin{enumerate}
    \item[\textup{(1)}] Let $P$ be a projective space over $k$ and $X \subseteq P$ a complete intersection defined by $f_1,\ldots,f_l$.
    If $\dim X \geq 1$, then the section ring of $X$ coincides with the coordinate ring of $X$, namely the quotient of a polynomial ring by $f_1,\ldots,f_l$. 
    
    Indeed, set $X_0:=P$ and for $1 \leq i \leq l$, let $X_i$ be the complete intersection in $P$ defined by $f_1,\ldots,f_i$.
    For each $h \in \Z$, there is an exact sequence
    \[
    \xymatrix{
     0 \ar[r] & \sO_{X_{i-1}}(-d_i+h) \ar[r]^-{\cdot f_i} & \sO_{X_{i-1}}(h) \ar[r] & \sO_{X_i}(h) \ar[r] & 0 
    }
    \]
    where $d_i=\deg(f_i)$.
    By induction on $i$, this yields
    \[
    H^{j}(X_{i},\sO_{X_{i}}(h))=0
    \]
    for all $0 \leq i \leq l$, $1 \leq j \leq \dim X_i-1$, and $h \in \Z$.
    In particular, the natural restriction map
    \[
    H^0(X_{i-1},\sO_{X_{i-1}}(h)) \longrightarrow H^0(X_i,\sO_{X_i}(h))
    \]
    coincides with the quotient map modulo $(f_1,\ldots,f_l)$ for all $h \in \Z$, as claimed.

    \item[\textup{(2)}] Let $X_i \hookrightarrow \P^{n_i}_k$ be a complete intersection defined by $f_{i,1},\ldots,f_{i,l_i}$ with $\dim X_i \geq 1$ for $i=1,\ldots,m$.
    Set 
    \[
    X:=X_1 \times \cdots \times X_m \hookrightarrow \P^{n_1}_k \times \cdots \times \P^{n_m}_k=:P.
    \]
    Then the section ring of $X$ coincides with the coordinate ring of $X$, which is the quotient of the polynomial ring by $(f_{i,j} \mid 1 \leq i \leq m,\ 1 \leq j \leq l_i)$.
    This follows from K\"unneth's formula together with the argument in (1).
\end{enumerate}
\end{rmk}

\begin{rmk}\label{wrmk:projectively normality}
Let $k$ be a field.
\begin{enumerate}
    \item Let $P$ be a well-formed weighted projective space over $k$ and $X \subseteq P$ a hypersurface defined by $f$.
    Assume that the regular locus of $P$ contains all codimension-one points of $X$. 
    Then the coordinate ring of $X$ coincides with the section ring of $X$, namely the quotient of a polynomial ring by $f$.
    The proof is essentially the same as in Remark \ref{rmk:projectively normality}.
    In particular, the vanishing of cohomology groups follows from an argument similar to \cite[Theorem~1.4]{Dol}.

    \item Let $X_i \hookrightarrow P_i$ be a hypersurface in a well-formed weighted projective space over $k$ defined by $f_i$ for $i=1,\ldots,m$.
    Assume that the regular locus of each $P_i$ contains all codimension-one points of $X_i$.
    Set 
    \[
    X:=X_1 \times \cdots \times X_m \hookrightarrow P_1 \times \cdots \times P_m=:P.
    \]
    Then the coordinate ring of $X$ coincides with the section ring of $X$, which is the quotient of the polynomial ring by $(f_1,\ldots,f_m)$.
    Indeed, by the K\"unneth formula and the argument in (1), the natural map
    \[
    \bigoplus_{h \in \Z^m} H^0(U,\sO_P(h)) \longrightarrow \bigoplus_{h \in \Z^m} H^0(X \cap U,\sO_X(h))
    \]
    is surjective, where $U$ denotes the regular locus of $P$.
    Since $U$ contains all codimension-one points of $X$, the claim follows.

    \item Consider the hypersurface
    \[
    X=\{f:=z^2+x^2y+xy^2=0\} \subseteq \P(2:2:3)
    \]
    over $\var{\F}_2$.
    Since $k[x,y,z]/(f)$ is normal, $X$ is normal as well.
    As $X$ is a hypersurface, the regular locus of $P$ contains all codimension-one points of $X$.
    Therefore the section ring of $X$ coincides with its coordinate ring.
    Moreover, $X$ is one-dimensional and $-K_X$ is ample by the adjunction formula, so $X$ is a projective line and hence $F$-split.
    However, $k[x,y,z]/(f)$ is not $F$-split by Fedder's criterion.
    This shows that, in Proposition \ref{wprop:multigraded algebra and proj}, the assumption that the weighted projective space is well-formed cannot be omitted.
    Furthermore, $X$ can also be realized as $\{f=w=0\}$ in the well-formed weighted projective space $\P(2:2:3:1)$ over $\var{\F}_2$.
    Thus it is also necessary to assume that the regular locus of a weighted projective space contains all codimension-one points of $X$.
\end{enumerate}
\end{rmk}

\begin{prop}\label{prop:multigraded ring versus local ring}
Let $S$ be an $F$-finite $\mathbb{Z}_{\geq 0}^m$-graded ring such that $S_0$ is a field $k$ of positive characteristic. 
Let $\m:=\bigoplus_{h \neq 0} S_h$.
Then we have $\sht(S)=\sht(S_\m)$.
\end{prop}

\begin{proof}
By the proof of Proposition \ref{wprop:multigraded algebra and proj}, we can see that $\Hom_{W_n(S)}(F_*W_n(S),S)$ has a multigraded $W_n(S)$-module structure.
Thus, the surjectivity of $\Hom_{W_n(S)}(F_*W_n(S),S) \to S$ is equivalent to that of $\Hom_{W_n(S_\m)}(F_*W_n(S_\m),S_\m) \to S_\m$, and we conclude that $\sht(S)=\sht(S_\m)$ by Remark \ref{rmk:evaluation}.
\end{proof}

\subsection{Fedder type criteria for projective varieties}
In this subsection, we introduce Fedder type criteria for quasi-$F$-splitting of projective varieties.

\begin{conv}\label{conv:graded ring}
In this subsection, $k$ is an $F$-finite field of characteristic $p>0$ and $S:=k[x_1,\ldots,x_N]$ is a polynomial ring with a $\mathbb{Z}_{\geq 0}^m$-graded structure.
We assume that the degree of $x_i$ is non-zero for all $i$, which is denoted by $\mu_i$
We define $\mu$ by $\mu:=\sum \mu_i$.
Let 
\[
\m:=\bigoplus_{h \in \mathbb{Z}_{\geq 0}^m \backslash \{0\}} S_h =(x_1,\ldots,x_N).
\]
Let $R:=S_\m$. Then $(R,\m,k)$ is a regular local ring, where we denote the extension of $\m$ in $R$ by $\m$ by abuse of notation.
Let $y_1,\ldots,y_s$ be a $p$-basis of $k$. 
We denote the elements of 
\[
\{y_1^{j_1}\cdots y_s^{j_s}x_1^{i_1} \cdots x_N^{i_N} \mid 0 \leq j_1,\ldots,j_s, i_1,\ldots,i_N \leq p-1\}
\]
by $v_1, \ldots, v_{d}$.
Then $F_{*}v_1, \ldots, F_*v_{d}$ is a $p$-basis of $F_*S$ over $S$.
We may assume that
\[
v:=v_d= (y_1\cdots y_s x_1 \cdots x_N)^{p-1}.
\]
The dual basis is denoted by $u_1,\ldots,u_d$ and $u:=u_d$.
We note that $\deg(v)=(p-1)\mu$.
Let $a \in S$. We take the monomial decomposition $a=\sum^m_{i=1} M_i$.
We define $\Delta_1(a) \in S$ by 
\[
(0,\Delta_1(a))=(a,0)-\sum (M_i,0)\ \hspace{0.2cm} \text{in} \hspace{0.1cm} W_2(S).
\]
We can compute that 
\[
\Delta_1(a)=\sum_{\substack{0 \leq \alpha_1, \ldots,\alpha_m \leq p-1 \\ \alpha_1+\cdots+\alpha_m=p}} \frac{1}{p} \binom{p}{\alpha_1, \ldots ,\alpha_m}(M_1)^{\alpha_1} \cdots (M_m)^{\alpha_m}.
\]

\end{conv}



\begin{thm}\label{thm:Fedder's criterion for projective varieties}
Let $P_1,\ldots,P_m$ be projective spaces over a perfect field $k$ of characteristic $p>0$, and set $P \coloneq P_1\times \cdots \times P_m$.
Let $S$ denote the section ring of $P$; then $S$ is a polynomial ring $S=k[x_1,\ldots,x_N]$ with its natural multigraded structure.
Let $X$ be a closed subscheme of $P$ defined by homogeneous regular elements $f'_1,\ldots,f'_l \in S$.
Assume that the section ring of $X$ coincides with the coordinate ring of $X$ (see Definition \ref{wdefn:notation for multiprojective space}).

Define the $S$-module homomorphism
\[
\theta \colon F_*S \longrightarrow S, \qquad F_*a \longmapsto u\bigl(F_*(\Delta_1(f^{p-1})a)\bigr),
\]
and set $I_1:=(I^{[p]}:I)$.
Inductively, define
\[
I_{s+1}:=\theta(F_*I_s \cap \Ker(u))+I_1.
\]

Then the following hold:
\begin{itemize}
    \item[\textup{(1)}] 
    \[
\sht(S/(f'_1,\ldots ,f'_l)) =\sht(X)=\inf\{\,n \mid I_n \nsubseteq \m^{[p]}\,\}.
    \]

    \item[\textup{(2)}] Suppose that the degree of $f:=f'_1 \cdots f'_l$ equals $N$.
We set $f_1:=f^{p-1}$, and
    \[
    f_n:=f^{p-1}\Delta_1(f^{p-1})^{1+p+\cdots+p^{n-2}}
    \]
    for $n \geq 2$.
    Then
    \[
 \sht(S/(f'_1,\ldots ,f'_l)) =   \sht(X)=\inf\{\,n \mid f_n \notin \m^{[p^n]}\,\}.
    \]
\end{itemize}
\end{thm}

\begin{proof}
By Theorem~\ref{CY-fedder}, it suffices to show that 
\[
\sht(S/(f'_1,\ldots ,f'_l))=\sht(X).
\]
By Proposition \ref{wprop:multigraded algebra and proj}, if the section ring of $X$ coincides with its coordinate ring, the quasi-$F$-split height of $X$ equals that of the coordinate ring $S/(f'_1,\ldots,f'_l)$.
\end{proof}

\begin{rmk}\label{rmk:projective normality remark}
By Remark \ref{rmk:projectively normality}, the section ring of $X$ coincides with the coordinate ring of $X$ if $X=X_1 \times \cdots \times X_m$ for some complete intersections $X_i \hookrightarrow P_i$ with $\dim X_i \geq 1$ for every $i$.
\end{rmk}

\begin{thm}\label{thm:Fedder's criterion for projective varieties in weighted case}
Let $P$ be a well-formed weighted projective space (see Definition \ref{wdefn:well-formed}) over a perfect field $k$ of characteristic $p>0$.
Denote by $S$ the section ring of $P$, which is the polynomial ring $S:=k[x_1,\ldots,x_N]$ with a suitable grading.
Let $X$ be a hypersurface in $P$ defined by a homogeneous regular element $f \in S$ with $\dim(X) \geq 1$.
Assume that every codimension-one point of $X$ lies in the regular locus of $P$.
Then the following hold:
\begin{itemize}
    \item[\textup{(1)}] Define the sequence of ideals $\{I_n\}$ as in Theorem \ref{thm:Fedder's criterion for projective varieties} using $f$. Then
    \[
    \sht(X)=\inf\{\,n \mid I_n \nsubseteq \m^{[p]}\,\}.
    \]
    \item[\textup{(2)}] Assume that $\deg(f)=N$ and define the sequence $\{f_n\}$ as in Theorem \ref{thm:Fedder's criterion for projective varieties} using $f$. Then
    \[
    \sht(X)=\inf\{\,n \mid f_n \notin \m^{[p^n]}\,\}.
    \]
\end{itemize}
\end{thm}

\begin{proof}
By Remark~\ref{wrmk:projectively normality}, the section ring of \(X\) coincides with its coordinate ring. 
Therefore, the assertions follow from \cite[Corollary~5.4 and Theorem~5.8]{KTY1} together with Proposition \ref{wprop:multigraded algebra and proj}.
\end{proof}

\begin{eg}[\textup{Fermat type, cf.~\cite[Theorem 5.1]{GK2}}]\label{eg:Fermat type}
Let $N \geq 4$, $S:=k[x_1,\ldots,x_N]$, and
\[
f:=x_1^N+x_2^N+ \cdots +x_N^N.
\]
We show 
\[
\sht ( \Proj (S/f) ) =
\sht(S/f)=
\begin{cases}
1     & p \equiv 1 \mod N  \\
\infty & p \nequiv 1 \mod N.
\end{cases}
\]
by using Theorem \ref{thm:Fedder's criterion for projective varieties} as follows:
First, we consider the case where $p-1$ is divisible by $N$.
Then $f^{p-1}$ contains the term
\[
\binom{p-1}{\frac{p-1}{N},\ldots,\frac{p-1}{N}} x_1^{p-1} \cdots x_N^{p-1},
\]
and thus $f^{p-1} \notin \m^{[p]}$ and $\sht(S/f) =1$.

Next, consider the remaining case.
By Corollary \ref{cor:non-q-split criterion}, it is enough to show that
\[
f^{p-2} f^{p(p-2)}\Delta_1(f) \in \m^{[p^2]}.
\]
We take a term $x_1^{pNa_1}\cdots x_N^{pNa_{N}}$ in $f^{p(p-2)}$, then $a_1+\cdots+a_N=p-2$.
If it is not contained in $\m^{[p^2]}$, then $Na_1,\ldots,Na_N \leq p-1$.
Since $p \nequiv 1 \mod N$, we have $Na_1,\ldots,Na_N \leq p-2$.
Combining the condition $a_1+ \cdots +a_N=p-2$, we have $Na_i=p-2$ for all $i=1,\ldots,N$, and in particular, $p \equiv 2 \mod N$.
Therefore, if a term $x_1^{b_1} \cdots x_N^{b_N}$ appearing in $f^{p-2}f^{p(p-2)}\Delta_1(f)$ is not contained in $\m^{[p^2]}$, then $b_i=p(p-2)+N(c_i+d_i) \leq p^2-1$, where
\[
c_1+\cdots+c_N=p-2,\ \mathrm{and}\ d_1+\cdots+d_N=p.
\]
Thus we have $N(c_i+d_i) \leq 2p-1$ for all $i$.
Since $p \equiv 2 \mod N$ and $N \geq 4$, we have $2p \nequiv 1,\ 2 \mod N$, thus $n(c_i+d_i) \leq 2p-3$.
It contradicts the above summation condition.
Therefore, we have $f^{p-2}f^{p(p-2)}\Delta_1(f) \in \m^{[p^2]}$ and $\sht(S/f)=\infty$ by Theorem~\ref{thm:Fedder's criterion for projective varieties}.
\end{eg}

\begin{eg}\label{eg:cusp}
Let $S:=k[x,y,z]$ and $f:=x^3+y^2z$.
By Theorem \ref{thm:Fedder's criterion for projective varieties}, we show that both $S/f$ and $\Proj(S/f)$ are not quasi-$F$-split as follows:

First, we prove $f^{p-1} \in \m^{[p]}$.
Take an integer $0 \leq i \leq p-1$ and consider the monomial
\[
(x^3)^i(y^2z)^{p-1-i}.
\]
If it is not contained in $\m^{[p]}$, then we have
\[
3i \leq p-1,\ \text{and}\ 2(p-1-i) \leq p-1,
\]
and we obtain a contradiction.

Next,  we assume that $p > 5$.
Then one can see that $\Delta_1(f^{p-1}) \in \m^{[p^2]}$.
Indeed, if $\Delta_1(f^{p-1}) \notin \m^{[p^2]}$, then we have 
\[
p^2-6p+5 \leq 0,
\]
a contradiction.
Note that $\Delta_1(f^{p-1})$ is a $k$-linear combination of monomials, each of which is a  product of $p(p-1)$ (not necessarily distinct) monomials of $f$.
Thus, Corollary \ref{cor:non-q-split criterion} shows that $S/f$ is not quasi-$F$-split.

Finally, we assume that $p \leq 5$.
Then we have
\begin{equation*}  \label{eq: cases f}
\Delta_1(f^{p-1}) \equiv
    \begin{cases}
        ax^{24}y^{24}z^{12}  &  p=5,   \\
        bx^6y^{8}z^4        &  p=3    \\
        cx^3y^2z          &  p=2
    \end{cases}
    \mod \m^{[p^2]},
\end{equation*}
where $a,b,c\in k$.
An easy computation shows that
\[
\Delta_1(f^{p-1})^{1+p} \in \m^{[p^3]}\ \text{and}\ f^{p-1}\Delta_1(f^{p-1}) \in \m^{[p^2]},
\]
and in particular, 
\[
f_n:=f^{p-1} \Delta_1(f^{p-1})^{1+p+\cdots+p^{n-2}} \in \m^{[p^n]}.
\]
Now, we conclude by Theorem~\ref{thm:Fedder's criterion for projective varieties} that $S/f$ is not quasi-$F$-split.
\end{eg}

\section{Applications of Fedder type criteria}

\subsection{Inversion of adjunction}
Let $X$ be a hypersurface in a projective space.
By the original Fedder's criterion, one readily verifies that if a general hyperplane section of $X$ is $F$-split, then $X$ is $F$-split as well, i.e.,\ the inversion of adjunction holds for $F$-splitting.
In this subsection, we discuss an analogue in the quasi-$F$-split setting. First, we show that a positive result holds when the hyperplane section is Calabi--Yau.
The key ingredient in the proof is a Fedder type criterion in the Calabi--Yau setting \cite[Theorem~5.8]{KTY1}.
Next, we present an example showing that, in general, inversion of adjunction fails for quasi-$F$-splitting.

\begin{prop}\label{cor:hypersurface inversion of adjunction}
Let $k$ be a perfect field of positive characteristic.
Let $X \subseteq \P^{N-1}_k$ be a hypersurface defined by a homogeneous polynomial $f$ of degree $m \leq N$.
Let $Y$ be a complete intersection of $X$ with $(N-m)$ hyperplanes, and assume that $\dim Y = m-2  >0$.
Then $\sht(Y) \geq \sht(X)$.
\end{prop}

\begin{proof}
The claim follows from \cite[Lemma~5.14]{KTY1} together with Theorem~\ref{thm:Fedder's criterion for projective varieties}.
\end{proof}

\begin{prop}\label{cor:hypersurface inversion of adjunction weighted}
Let $k$ be a perfect field of positive characteristic. 
Let $X \subseteq P:=\P(\mu_1,\ldots,\mu_N)$ be a hypersurface in a well-formed weighted projective space defined by a homogeneous polynomial $f$ of degree $\sum_{i=1}^N \mu_i-1$.
Let $Y \subseteq P$ be the complete intersection defined by $f$ and a linear form $g$ with $\deg(g)=1$.
Assume moreover that $Y$ is normal. 
Then $\sht(Y) \geq \sht(X)$.
\end{prop}

\begin{proof}
If $\dim(Y)=0$, the claim follows from an argument similar to \cite[Lemma~5.14]{KTY1}.
Thus we may assume $\dim(Y) \geq 1$.
Let $S$ be the section ring of $P$.
Then $S$ is the polynomial ring $k[x_1,\ldots,x_N]$ with $\deg(x_i)=\mu_i$.
After a change of variables, we may assume $g=x_1$ (and $\mu_1=1$).
By \cite[Lemma~5.14]{KTY1}, we have $\sht(S/(f,g)) \geq \sht(S/f)$.
By Proposition~\ref{wprop:multigraded algebra and proj}, we have $\sht(S/f) \geq \sht(X)$.
Hence it remains to show $\sht(Y)=\sht(S/(f,g))$, 
which follows 
from Theorem~\ref{thm:Fedder's criterion for projective varieties in weighted case}.
\end{proof}

The following example shows that the inversion of adjunction-type result fails for quasi-$F$-splitting.

\begin{eg}\label{eg:counterexample to inversion of adjunction}
Let $S:=k[x,y,z,w,u,s]$ and $p=2$.
Let $g:=xys^2+zwu^2+y^3w+x^3z$.
Then we have
\begin{itemize}
    \item[\textup{(1)}] $\sht(S/g)=\infty$ but
    \item[\textup{(2)}] $\sht(S/(s,g))=2$.
\end{itemize}

First, we show (1).
Let $J:=(ys^2+x^2z,zu^2+y^3)$. Then, we have
\[
J \supseteq \theta_g(F_*J \cap v^\perp) +(g),
\]
hence $\sht(S/g)=\infty$ by Corollary \ref{cor:Fedder's criterion for quasi-F-splitting}.

Next, we show (2). We have $S/(g,s) \cong k[x,y,z,w,u]/f$, where 
\[
f:=zwu^2+y^3w+x^3z.
\]
We have $u(F_* (zwuf ))=0$ and $\theta_f(F_* (zwuf))=xyzwu \notin \m^{[p]}$, hence $\sht(S/(s,g))=2$.
\end{eg}

\begin{eg}
\label{eg:strange graded example}
Let $S:=k[x,y,z,w,u,s]$ and $p:=2$.
Define
\begin{eqnarray*}
g := xys^2+zwu^2+z^3u+y^3w+x^3z.
\end{eqnarray*}
Note that $g$ is a homogeneous polynomial of degree $4$.
Then we have
\begin{itemize}
    \item[\textup{(1)}] $\sht(S/g)=3$, but
    \item[\textup{(2)}] $\sht(S/(s,g))=2$. 
\end{itemize}

We show (1).
We first show that $\sht(S/g) \geq 3$.
Clearly, we have $g \in \m^{[2]}$.
We can see that $u(F_{\ast}g\Delta_{1}(g)\cdot x^{i_{1}}y^{i_{2}}z^{i_{3}}w^{i_{4}}u^{i_{5}}s^{i_{6}})$ is not contained in $\m^{[2]}$ if and only if 
\[
(i_{1},i_{2},i_{3},i_{4},i_{5},i_{6}) \equiv (0,0,1,1,1,1) \pmod 2.
\]
Moreover, we have  
\[
u(F_{\ast}(zwusg\Delta_{1}(g)))=zwu^2s+z^3us+y^3ws+x^3zs+xyzwu,
\]
and 
\[
F_{\ast}g = sF_{\ast}(xy)+uF_{\ast}(zw)+xF_{\ast}(xz)+yF_{\ast}(yw)+zF_{\ast}(zw).
\]
Suppose that 
\[
u(F_{\ast}(g\Delta_{1}(g)a)) \notin \m^{[2]},
\]
where $a\in R$.
Moreover, we decompose $a$ as
\[
a =  \sum_{0 \leq i_{j} \leq p-1} a_{i_{1},i_{2},i_{3},i_{4},i_{5},i_{6}} F_{\ast}( x^{i_{1}}y^{i_{2}}z^{i_{3}}w^{i_{4}}u^{i_{5}}s^{i_{6}}).
\]
By assumption, $a_{0,0,1,1,1,1}$ has a non-zero constant term or a non-zero $s$-term.
Therefore,
\[
u(ag)=a_{0,0,1,1,1,1}s+a_{1,1,0,0,1,1}u+a_{0,1,0,1,1,1}x+a_{1,0,1,0,1,1}y+a_{1,1,0,0,1,1}z
\]
also has a non-zero $s$-term or a non-zero $s^2$-term.
Thus we have $I_{2}(g) \subseteq \m^{[2]}$.
We next show that $I_{3}(g) \nsubseteq \m^{[2]}$.
We can see that 
\[
u(F_* (z^2u^2(xyz^2us+xywu^2s)g ))=0, 
\]
\[
u(F_* (\theta_{g}(F_*(z^2u^2(xyz^2us+xywu^2s)g )))) =0,
\]
and
\[
\theta_{g}(F_* (\theta_{g}(F_* (z^2u^2(xyz^2us+xywu^2s)g))) ) = xyzus+x^2y^2z \notin \m^{[2]}.
\]
Therefore, we have the claim.

Finally, (2) follows similarly to Example \ref{eg:counterexample to inversion of adjunction}.
\end{eg}

\subsection{Fiber products}
In this subsection, we compute the quasi-$F$-split height of a fiber product of projective varieties.
First, we prove that if all but one factor are $F$-split, then the quasi-$F$-split height of the fiber product coincides with that of the remaining factor.
Next, we show that a fiber product of non-$F$-split complete intersection varieties is never quasi-$F$-split.
We note that, after the first version of this paper was completed, the contents of this chapter were generalized to a more general setting by \cite{YobukoHodgeWitt}, by a completely different proof.

\begin{prop}\label{prop:fiber products}
Let $X$ and $Y$ be projective varieties over a field $k$ of positive characteristic such that $H^0(X,\sO_X)\cong H^0(Y,\sO_Y)\cong k$.
If $\sht(X)=n$ and $\sht(Y)=1$, then $\sht(X \times_k Y)=n$.
\end{prop}

\begin{proof}
Since $H^0(X,\sO_X)\cong H^0(Y,\sO_Y)\cong k$, we may embed $X$ and $Y$ as projectively normal subvarieties of some $\P^N_k$.
Let $S_X$ and $S_Y$ denote the section rings of $X$ and $Y$, respectively.
By K\"unneth's formula, the section ring of $X \times Y$ inside $\P^N \times \P^N$ coincides with the coordinate ring $S_X \otimes_k S_Y$.
Thus, by Proposition \ref{wprop:multigraded algebra and proj}, it suffices to show that the quasi-$F$-split height of $S_X \otimes_k S_Y$ is equal to $n$.

Let $S_{\P^N}:=k[x_0,\ldots,x_N]$ be a polynomial ring, and let $R_{\P^N}:=S_{\P^N, (x_1,\ldots,x_N)}$ be the localization.
Let $R_X$ and $R_Y$ denote the localizations at the origin of $S_X$ and $S_Y$, respectively.
Then $\sht(R_X)=n$, $\sht(R_Y)=1$, and by Proposition \ref{prop:multigraded ring versus local ring} we have
\[
\sht(R_X \otimes R_Y)=\sht(S_X \otimes S_Y).
\]
Since $R_X$ and $R_Y$ are quotients of $R_{\P^N}$, we may write $R_X \cong R_{\P^N}/I$ and $R_Y \cong R_{\P^N}/J$ for some ideals $I,J \subseteq R_{\P^N}$.

As in Convention \ref{conv:graded ring}, we define $u$ for the ring $S_{\P^N}$.
We also define $\widetilde{u}$ for $S_{\P^N} \otimes S_{\P^N} \simeq k[x_0, \ldots, x_{2N+1}]$ in a similar way.
By \cite[Theorem~4.8]{KTY1}, there exist $g_1,\ldots,g_n \in R_{\P^N}$ satisfying the required conditions for $I$.
Since $\sht(R_{\P^N}/J)=1$, there exists $h \in (J^{[p^n]}:J)$ such that $u^{n-1}(F^{n-1}_*h) \notin \m^{[p]}$.
Define $f_1,\ldots,f_n \in R_{\P^N} \otimes R_{\P^N}$ by
\[
f_i := g_i \otimes u^{n-i}(F^{n-i}_*h).
\]
We now verify that $f_1,\ldots,f_n$ satisfy the conditions of \cite[Theorem~4.8]{KTY1}.
Since $u(F_*g_i)=0$, it follows that $\widetilde{u}(F_*f_i)=0$ for all $i \geq 2$.
Moreover, because $g_1,\,u^{n-1}(F^{n-1}_*h) \notin \m^{[p]}$, we have $f_1 \notin \n^{[p]}$, where $\n$ denotes the maximal ideal of $R_{\P^N} \otimes R_{\P^N}$.

Next, fix $1 \leq s \leq n$.
For $x \in I$, we have
\[
\sum_{r=0}^{n-s} \widetilde{u}^r(F^r_*f_{r+s}\Delta_r(x\otimes 1))
= \sum_{r=0}^{n-s} u^r(F^r_*g_{r+s}\Delta_r(x)) \otimes u^{n-s}(F^{n-s}_*h) \in I^{[p^s]} \otimes R_{\P^N}. 
\]
Similarly, for $y \in J$ we obtain
\begin{align*}
\sum_{r=0}^{n-s} \widetilde{u}^r(F^r_*f_{r+s}\Delta_r(1\otimes y))
&= \sum_{r=0}^{n-s} u^r(F^r_*g_{r+s}) \otimes u^r\big(F^r_*u^{n-r-s}(F^{n-r-s}_*h)\Delta_r(y)\big) \\
&= g_s \otimes u^{n-s}(F^{n-s}_*hy^{p^{\,n-s}}) \in R_{\P^N} \otimes J^{[p^s]}.
\end{align*}
Therefore, $f_1,\ldots,f_n$ satisfy the conditions of \cite[Theorem~4.8]{KTY1}.
We conclude that the quasi-$F$-split height of $R_X \otimes R_Y$ is equal to $n$.
\end{proof}

\begin{prop}\label{prop:fiber product not quasi-F-split}
Let $X \hookrightarrow \P_k^n$ and $Y \hookrightarrow \P_k^m$ be complete intersections in projective spaces over a field $k$ with $H^0(X,\sO_X)\cong H^0(Y,\sO_Y)\cong k$.
If the quasi-$F$-split heights of $X$ and $Y$ are greater than $1$, then $X \times_k Y$ is not quasi-$F$-split.
\end{prop}

\begin{proof}
Let $S_X$ and $S_Y$ be the section rings of $X$ and $Y$, respectively.
Then $S_X \otimes_k S_Y$ is the section ring of $X \times_k Y$ in $\P^n \times_k \P^m$.
Note that $\sht(S_X), \sht(S_Y) \geq 2$.
Let $g_1,\ldots,g_s$ and $h_1,\ldots,h_t$ be homogeneous regular sequences defining $X$ and $Y$, respectively, and set $g:=g_1 \cdots g_s$ and $h:=h_1 \cdots h_t$.
Since $X$ and $Y$ are not $F$-split, we have 
\[
g^{p-1} \in \m_{S_X}^{[p]} \quad\text{and}\quad h^{p-1} \in \m_{S_Y}^{[p]}.
\]
Therefore,
\[
\Delta_1(g^{p-1}h^{p-1})
= g^{p(p-1)}\Delta_1(h^{p-1}) + h^{p(p-1)}\Delta_1(g^{p-1})
\in \m^{[p^2]},
\]
where $\m:=\m_{S_X}\otimes 1 + 1 \otimes \m_{S_Y}$.
By Corollary \ref{cor:non-q-split criterion}, it follows that $X \times Y$ is not quasi-$F$-split.
\end{proof}

\subsection{Smoothness of genus one fibrations}
In this subsection, we study a fibration whose relative canonical divisor is trivial.
For such a fibration, we prove that if the generic fiber is quasi-$F$-split and 
isomorphic to a complete intersection in some projective space over $K(Y)$, then a general fiber is also quasi-$F$-split.
Note that this fails in general without the assumption on the relative canonical divisor (see Example \ref{eg:wild conic bundle}).
As a corollary, we prove that quasi-$F$-split surfaces have no quasi-elliptic fibration.

\begin{lem}\label{lem:very ample descent}
Let $k \subseteq K$ be a field extension.
Let $X$ be a proper scheme over $k$ such that $H^0(X,\sO_X)=k$.
Let $X_K:=X \times_k K$ and $\pi\colon X_K \to X$ the natural projection.
Let $L$ be an invertible sheaf on $X$.
If $L_K:=\pi^*L$ is very ample and $X_K$ is projectively normal with respect to the embedding given by $L_K$, then $L$ is very ample and $X$ is projectively normal with respect to the closed immersion defined by $L$.
\end{lem}

\begin{proof}
We define a graded ring $B$ by
\[
B:=\bigoplus_{n \geq 0} H^0(X,L^n).
\]
Then the base change $B_K := B \otimes_k K$ is 
\[
B_K \cong \bigoplus_{n \geq 0} H^0(X_K,L_K^n).
\]
By the assumption on $X_K$, we have a surjective graded homomorphism
\[
K[x_1,\ldots,x_N] \to B_K.
\]
In particular, the dimension of $H^0(X_K,L_K)$ over $K$ is $N$, and thus the dimension of $H^0(X,L)$ over $k$ is $N$.
Replacing with a suitable basis,
we may assume that we have the commutative diagram
\[
\xymatrix{
k[x_1,\ldots,x_N] \ar[r]^-{} \ar[d] \ar@{}[rd]|{\circlearrowright} & K[x_1,\ldots,x_N] \ar@{->>}[d]^-{} \\
B \ar[r]_-{}  & B_K.
}
\]
Furthermore, since the diagram is Cartesian and $k \to K$ is faithfully flat, the left vertical map is also surjective.
Therefore, it defines a closed immersion $X \hookrightarrow \P^N_k$,
with $X$ projectively normal for this embedding.
\end{proof}

\begin{cor}\label{cor:base change, Calabi-Yau variety}
Let $k \subseteq K$ be a field extension.
Let $X$ be a proper scheme over $k$ such that $H^0(X,\sO_X)=k$ and $\dim X \geq 1$.
Let $X_K:=X \times_k K$.
Suppose that $X_K\subseteq\P^N_K$ is a complete intersection.
In addition, we assume that $\sO_{X_K}(1)$ descends to $X$ and $\omega_{X_K} \cong \sO_{X_K}$.
Then we have $\sht(X)=\sht(X_K)$.
\end{cor}

\begin{proof}
By assumption, there exists a line bundle $L$ on $X$ such that $\pi^*L \cong \sO_{X_K}(1)$.
We denote $L$ by $\sO_X(1)$.
We consider the section rings
\[
B:=\bigoplus_{n}H^0(X,\sO_X(n)) \to \bigoplus_{n} H^0(X_{K},\sO_{X_K}(n))=:B_K.
\]
We note that $B_K=B \otimes_k K$.
By the proof of Lemma \ref{lem:very ample descent}, we have the Cartesian diagram
\[
\xymatrix{
k[x_0,\ldots,x_N] \ar[r]^-{} \ar@{->>}[d] \ar@{}[rd]|{\circlearrowright} & K[x_0,\ldots,x_N] \ar@{->>}[d]^-{} \\
B \ar[r]_-{}  & B_K.
}
\]
Since $\dim X_K \geq 1$ and $X_K$ is a complete intersection, the kernel of the right vertical map is generated by a regular homogeneous sequence.
In particular, the kernel of the left vertical map is generated by a homogeneous regular sequence $f'_1,\ldots,f'_m$.
Furthermore, since $\omega_{X_K}$ is trivial, the degree of $f:=f'_1 \cdots f'_m$ is $N$.
Therefore, by \cite[Theorem~5.13]{KTY1}, we have $\sht(B)=\sht(B_K)$.
Since $X$ and $X_K$ are projectively normal, we have $\sht(X)=\sht(X_K)$.
\end{proof}

\begin{rmk}
Corollary \ref{cor:base change, Calabi-Yau variety} does not hold for non-Calabi--Yau cases as in Example \ref{eg:wild conic bundle}.
\end{rmk}

\begin{cor}\label{cor:fiber space}
Let $k$ be an algebraically closed field of characteristic $p>0$.
Let $\pi \colon X \to Y$ be a morphism satisfying $\pi_{*}\sO_X=\sO_Y$.
We assume that the generic fiber has a trivial canonical divisor and is a complete intersection in a projective space over $K(Y)$.
Then quasi-$F$-split heights of a general fiber, the generic fiber, and the geometric generic fiber coincide.
In particular, if $X$ is quasi-$F$-split, then so is a general fiber.
\end{cor}

\begin{proof}
By shrinking $Y$, we may assume that $Y$ is an affine scheme $\Spec B$ and $X$ is a complete intersection in $\P^N_B$.
The defining equation of $X$ is denoted by $f_1,\ldots,f_m \in B[x_0,\ldots,x_N]$, which is a regular sequence.
After shrinking $\Spec B$ further, we may assume that $f_1(s),\ldots,f_m(s)$ is a regular sequence in $\kappa(s)[x_0,\ldots,x_N]$ for any point $s \in \Spec B$.
Let $G:=f_1\cdots f_m\in B[x_0,\ldots,x_N]$.
Since $\omega_X \cong \sO_X$, the degree of $G$ is $N+1$.
We take the monomial decomposition 
\[
G^{p-1}=\sum_{i=1}^{m} b_i M_i
\]
of $G^{p-1}$ over $B$,
where every $M_i$ is a monomial and $b_i \in B$.
Furthermore, we define
\[
\widetilde{\Delta}_1(G^{p-1}):=\sum_{\substack{0 \leq \alpha_1,\ldots,\alpha_m \leq p-1 \\ \alpha_1+\cdots+\alpha_m=p}} \frac{1}{p} \binom{p}{\alpha_1,\ldots,\alpha_m} (b_1M_1)^{\alpha_1} \ldots (b_mM_m)^{\alpha_m}.
\]
We take a geometric point $\var{s}$ of $\Spec B$, then
\[
G^{p-1}(\var{s})=\sum_{i=1}^m b_i(\var{s})M_i
\]
is a $p$-monomial decomposition since $b_i(\var{s})$ has a $p$-th root in $\kappa(\var{s})$.
Therefore, we obtain 
\[
\widetilde{\Delta}_1(G^{p-1})(\var{s}) \equiv \Delta_1(G(\var{s})^{p-1})\ \mod F(\kappa(\var{s})[x_0,\ldots,x_N]),
\]
where the right-hand side is defined as in Convention \ref{conv:graded ring}.
Furthermore, since both sides are homogeneous
of degree $p(p-1)(N+1)$, by \cite[Theorem~5.8]{KTY1}, if we define 
\[
G_n:=G^{p-1}\widetilde{\Delta}_1(G^{p-1})^{1+p+\cdots+p^{n-2}},
\]
then we have
\[
\sht(X_{\var{s}})=\mathrm{inf}\{n \mid G_n(\var{s}) \notin \m^{[p^n]}\},
\]
where $X_{\var{s}}$ is the geometric fiber of $X$.
Furthermore, by \cite[Lemma~5.7]{KTY1}, if we define $g_i \in B$ by the coefficient of $(x_0\cdots x_N)^{p^n-1}$ in $G_n$, then we have
\[
\sht(X_{\var{s}})=\mathrm{inf}\{n \mid G_n(\var{s}) \neq 0 \}.
\]
Therefore, if we put
\[
Y_{\geq h}:=V(g_1,\ldots,g_{h-1}),
\]
then for any geometric point $\var{s} \in Y$, $\sht(X_{\var{s}})\geq h$ if and only if $\var{s} \in Y_{\geq h}$.
Furthermore, by \cite[Theorem~5.13]{KTY1}, we also have that for a point $s \in Y$, $\sht(X_s) \geq h$ if and only if $s \in Y_{\geq h}$.

Next, assume that $Y_{\geq h}=Y$ for every $h \in \Z_{>0}$.
Then every fiber has quasi-$F$-split height $\infty$ by the above observation.
Therefore, we obtain the desired result.
Otherwise, let $h_0$ be the minimal integer with $Y_{\geq h_0} \neq Y$.
Then, if we define $Y_0:=Y \backslash Y_{\geq h_0}$, then it is non-empty open and for every $s \in Y_0$, $\sht(X_s)=\sht(X_{\var{s}})=h_0-1$.
Therefore, the quasi-$F$-split height of a general fiber, the generic fiber, and the geometric generic fiber are $h_0-1$.

Finally, we assume that $X$ is quasi-$F$-split.
By definition, quasi-$F$-splitting is preserved under localization.
Therefore, the generic fiber is also quasi-$F$-split, and in particular, a general fiber is also quasi-$F$-split.
\end{proof}

In the above corollary, we need the assumption that the generic fiber is a complete intersection for technical reasons.
Therefore, we conjecture the following statement.

\begin{conj}
Let $\pi \colon X \to Y$ be a morphism of normal varieties satisfying $\pi_{*}\sO_X=\sO_Y$.
If the relative canonical divisor is trivial and $X$ is quasi-$F$-split, then a general fiber is quasi-$F$-split.
\end{conj}

\begin{rmk}
This conjecture does not hold in general when the relative canonical divisor is anti-ample (see Example \ref{eg:wild conic bundle}).
\end{rmk}

\begin{thm}\label{thm:no quasi-elliptic fibration}
Let $k$ be an algebraically closed field of characteristic $p>0$.
Let $\pi \colon X \to Y$ be a morphism of normal varieties satisfying $\pi_{*}\sO_X=\sO_Y$.
Suppose that the relative canonical sheaf is trivial, that the relative dimension of $\pi$ is one, and that a general fiber of $\pi$ is reduced.
If $X$ is quasi-$F$-split, then a general fiber of $\pi$ is smooth.
\end{thm}

\begin{proof}
By \cite[Corollary 1.8]{PW}, we may assume that $p=2$ or $3$.
We denote the function field of $Y$ by $K$ and the generic fiber by $C$.
Then $C$ is a regular one dimensional proper scheme over $K$ satisfying $\omega_C \cong \sO_C$ and $H^0(\sO_C)=K$.
We assume that $C$ is not smooth.
Then the base change of $C$ to $\var{K}$ is a rational curve with a cusp or a node, and 
since $C$ is regular, the singularity must be a cusp.
Since $X$ is quasi-$F$-split, so is $C$, and thus so is $C':=C \times_K K'$ by Corollary \ref{cor:base change}, where $K'$ is the separable closure of $K$.
By Example \ref{eg:cusp}, the base change of $C'$ to an algebraic closure is not quasi-$F$-split.
By \cite[Theorem~5.13]{KTY1}, it is enough to show that $C'$ is a complete intersection in a projective space.
By the computation in \cite[Section 2]{AKMMMP}, it is enough to show that there exists a line bundle $L$ on $C'$ such that $h^0(L)=3$ or $4$.
We take a closed point $P \in C'$.
Then, by the exact sequence
\[
0 \rightarrow \sO_{C'} \rightarrow \sO_{C'}(P) \rightarrow \kappa(P) \rightarrow 0,
\]
we have $\chi(\sO_{C'}(P))=\chi(\sO_{C'})+[\kappa(P) \colon K']$.
Since $\sO_{C'} \cong \omega_{C'}$, we have $\chi(\sO_{C'})=0$ and $\chi(\sO_{C'}(P))= [\kappa (P) \colon K']$.
Therefore, it is enough to show that there exists a closed point $P \in {C'}$ such that $[\kappa(P)\colon K']$ is $1$, $2$, $3$, or $4$. 
Let $Q \in C'$ be a non-smooth point of $C'$, which is unique since $C'$ is of genus one.
Then $R:=\sO_{C',Q}$ is a discrete valuation ring, geometrically integral and essentially of finite type over $K'$.
We denote the residue field of $R$ by $L$.
By \cite[Theorem 7.5]{NT},
we have
\[
\jac (R) = \frac{2p}{p-1},
\]
where $\jac (R)$ is the Jacobian number of $R$, defined by
\[
\dim_{K'} (R/\Fitt_{1}\Omega_{R/K'}^{1}) = [L:K'] \length_{R}(R/\Fitt_{1}\Omega_{R/K'}^{1}).
\]
If $p=3$, then
we have $\jac (R)=3$ and $[L:K'] = 1$ or $3$. 
On the other hand, if $p=2$, then
we have $\jac (R)=4$.
Therefore $[L:K'] \in \{1,2,4\}$, and the claim follows; see \cite[Theorem 11.8]{tanaka21} for another proof. 
\end{proof}

\begin{cor}\label{cor:genus one gibration surface}
A normal quasi-$F$-split surface over an algebraically closed field of positive characteristic has no quasi-elliptic fibration.
\end{cor}

\begin{proof}
It follows from Theorem \ref{thm:no quasi-elliptic fibration}.
\end{proof}

\subsection{Existence of Calabi--Yau varieties of arbitrary height in any characteristic}
In this subsection, we prove the existence of a Calabi--Yau hypersurface of quasi-$F$-split height (Artin--Mazur height) $h$ over $\var{\F}_p$ for any positive integer $h$ and any prime number $p$.
For the proof, we stratify the family of hypersurfaces by their quasi-$F$-split heights.

\begin{conv}\label{conv:unboundedness}
In this subsection, $k$ is an algebraically closed field of characteristic $p>0$ and $S:=k[x_1,\ldots, x_N]$ is a polynomial ring with the standard graded structure, that is, $\deg(x_i)=1$ for all $i$.
We assume that $N \geq 3$.
Let $\mathcal{M}$ be the set of all monic monomials of degree $N$.
Let $M := \# \mathcal{M}$,
$\mathcal{M}:=\{m_i \mid 1 \leq i \leq M\}$, and
\[
G:=\sum_{1 \leq i \leq M} a_i m_i \in k[x_j,a_i \mid 1 \leq j \leq N,\ 1 \leq i \leq M],
\]
where $a_1,\ldots,a_{M}$ are variables.
Since $k[x_j,a_i]$ is a free $k[a_i]$-module whose basis is monic monomials with respect to $x_1,\ldots,x_N$, 
every $H\in k[x_j,a_i]$ admits an expansion $H=\sum_{i} h_i H_i$, where $H_i \in S$ are monic monomials and $h_i \in k[a_i]$.
Then we define $\widetilde{\Delta}_1(H)$ by
\[
\widetilde{\Delta}_1(H):=\sum_{\substack{0 \leq \alpha_1,\ldots,\alpha_m \leq p-1 \\ \alpha_1+\cdots+\alpha_m=p}} \frac{1}{p} \binom{p}{\alpha_1,\ldots,\alpha_m} (h_1H_1)^{\alpha_1} \ldots (h_mH_m)^{\alpha_m}.
\]
Let $\mathcal{X}:=V_+(G)\subseteq\P^{N-1}_{\P^{M-1}_k}$ be the zero locus of $G$.
Then it defines the family
\[
\pi \colon \mathcal{X} \to P:=\P^{M-1}_k.
\]
Since the Hilbert polynomial is constant on fibers, the morphism $\pi$ is flat.
\end{conv}

\begin{lem}\label{lem:closed condition}
Let $h$ be a positive integer.
There exists a closed subset $P_{\geq h}$ of $P$ such that for every closed point $s$ of $P$,
the height of the fiber $\mathcal{X}_s$ of $s$ is at least $h$ if and only if $s \in P_{\geq h}$.
Moreover, $P_{\ge h}$ is cut out by a single equation inside $P_{\ge h-1}$.
\end{lem}
\begin{proof}
We define a sequence of elements $\{G_h\}_h$ by $G_1:=G^{p-1}$ and 
\[
G_n:=G^{p-1}\widetilde{\Delta}_1(G^{p-1})^{1+p+\cdots+p^{n-2}}
\]
for $n \geq 2$.
We take a point $s \in P$ and denote the fiber of $G_n$ and $G$ over $s$ by $G_n(s)$ and $G(s)$, respectively. Then $\mathcal{X}_s$ is the zero locus of $G(s)$ in $\mathbb{P}^{N-1}_k$.
By Theorem \ref{thm:Fedder's criterion for projective varieties} and Remark \ref{rmk:projective normality remark}, the height of $\mathcal{X}_s$ is at least $h$ if and only if $G_i(s) \in (x_1,\ldots,x_N)^{[p^h]}$ for all $1 \leq i \leq h-1$.
The latter condition is equivalent to the vanishing of the coefficient of $(x_1\cdots x_N)^{p^i-1}$ in $G_i(s)$ by \cite[Lemma~5.6]{KTY1}.
The coefficient of $G_i$ with respect to $(x_1 \cdots x_N)^{p^i-1}$ is denoted by $b_i \in k[a_1,\ldots,a_M]$.
By construction, $b_i$ is homogeneous of degree $p^i-1$.
Then $P_{\geq h}:=V_{+}(b_1,\ldots,b_{h-1})$ is the desired closed subset.
\end{proof}

\begin{cor}\label{cor:height bound for Calabi--Yau case}
There exists a positive integer $h_N$ such that, for every quasi-$F$-split hypersurface $X$ of degree $N$ in $\mathbb{P}^{N-1}_k$, the quasi-$F$-split height of $X$ is at most $h_N$.
\end{cor}

\begin{proof}
Since $P$ is Noetherian, Lemma \ref{lem:closed condition} yields $h_N$ with $P_{\ge h}=P_{\ge h_N+1}$ for all $h\ge h_N+1$.
For a hypersurface $X$ of degree $N$ in $\mathbb{P}^{N-1}_k$, if $X$ is quasi-$F$-split, then $X$ is a fiber over a point in $P_{\geq h}$, where $h$ is the height of $X$.
By the choice of $h_N$, the height of $X$ is at most $h_N$.
\end{proof}

\begin{thm}\label{thm:unbounded}
For every positive integer $h$, there exists a Calabi--Yau hypersurface $X_h$ over $k$ whose quasi-$F$-split height equals $h$.
\end{thm}

\begin{proof}
If $h=1$, then the Fermat type hypersurface in $\P^{p^3-2}$ is a smooth Calabi--Yau variety of height one (see Example \ref{eg:Fermat type}).
Therefore, we may assume that $h>1$.
Let $N:=p^h-1$.
We put
\[
\mathcal{M}':=\{m \in \mathcal{M} \mid m=x_1^N\ \mathrm{or}\ m \notin (x_1) \}.
\]
After reordering, we may assume that $m_1=x_1^N$ and $\mathcal{M}'=\{m_i \mid i \leq M' \}$ for some $M' \leq M$.
Taking the base change of $\mathcal{X}$ by the natural closed immersion $P':=\P^{M'-1} \hookrightarrow \P^{M-1}=P$, we define the family $\mathcal{X}' \to P'$ with the following Cartesian diagram:
\[
\xymatrix{
\mathcal{X}' \pullbackcorner  \ar[r]^-{\pi'} \ar[d]  & P' \ar[d] \\
\mathcal{X} \ar[r]_-{\pi} & P.
}
\]
Equivalently, $\mathcal{X}'$ is the zero locus of 
\[
G':=\sum_{1 \leq i \leq M'} a_i m_i 
\]
in $\P^{N-1} \times {\P^{M'-1}}$.
\begin{claim}\label{claim:smooth locus is non-empty}
There exists a nonempty open subset $U\subset P'$ over which $\pi'$ is smooth.
\end{claim}
\begin{claimproof}
It is enough to show that there exists a smooth fiber.
Since $\mathcal{M}'$ contains the monomials $x_1^N,\ldots,x_N^N$, the Fermat type hypersurface $x_1^N+\cdots+x_N^N=0$ appears as a fiber of $\pi'$, which shows the existence of a smooth fiber.
We note that $N=p^h-1$ is coprime to $p$.
\end{claimproof}

Since $P'$ is irreducible, it is enough to show that $P'_{h}:=(P_{\geq h} \backslash P_{\geq h+1})|_{P'}$ is a non-empty open subset.
Indeed, this shows $P'_{h}$ and $U$ in Claim \ref{claim:smooth locus is non-empty} has non-empty intersection, and a fiber over a point in the intersection gives a desired Calabi--Yau hypersurface.

First, we prove that $P'_h$ is open.
By Lemma \ref{lem:closed condition}, it is enough to show that there are no fibers of height $<h$.
We note that fibers of $\pi'$ are the zero loci of $g$ in $\mathbb{P}^{N-1}_k$, where $g$ is a $k$-linear sum of elements of $\mathcal{M}'$.
Therefore, in order to prove that $P'_h$ is open, it suffices to show the following claim:
\begin{claim}
Let $g$ be a $k$-linear sum of elements of $\mathcal{M}'$.
Then the height of $S/g$ is at least $h$.
\end{claim}
\begin{claimproof}
We can write $g=cx_1^N+g'$, where $c \in k$ and $g' \in k[x_2,\ldots,x_N]$.
Denote the height of $S/g$ by $l$. Then $g_l \notin \m^{[p^l]}$ by \cite[Theorem~5.8]{KTY1}, where $g_1:=g^{p-1}$ and 
\[
g_m:=g^{p-1}\Delta_1(g^{p-1})^{1+p+\cdots+p^{m-2}}
\]
for all $m \geq 2$.
Since the degree of $g_l$ is $(p^l-1)N$, we can deduce from $g_l \notin \m^{[p^l]}$ that the coefficient of $(x_1\cdots x_n)^{p^l-1}$ is non-zero by \cite[Lemma~5.6]{KTY1}.
On the other hand, 
the exponent of $x_1$ in any monomial appearing in $g_\ell$ is divisible by $N$ by the description $g=cx_1^N+g'$.
This shows that $p^l-1$ is divisible by $N$.
Since $N=p^h-1$, we have $l \geq h$.
\end{claimproof}

Finally, we prove that $P'_h$ is non-empty.
We choose monomials $B_1,\ldots,B_{p-1}$, $C_1,\ldots,C_{p-2}$ and $D_{i,j}$ for $1 \leq i \leq h$, $1 \leq j \leq p$ such that $B_1=x_1^N=x_1^{p^h-1}$, all such monomials are contained in $\mathcal{M}'$ and different from each other, and
\[
B \cdot C^{p+p^2+\cdots+p^{h-1}} \cdot D_1 \cdot D_2^{p} \cdots D_{h-1}^{p^{h-2}}=(x_1 \cdots x_N)^{p^h-1},
\]
where $B:=B_1 \cdots B_{p-1}$, $C:=C_1\cdots C_{p-2}$ and $D_i:=D_{i,1}\cdots D_{i,p}$.
Since the degree of $B$ (resp.~$C$, $D_i$) is $(p-1)N$ (resp.~$(p-2)N$, $pN$), 
it follows that the degree of the left-hand side in the above equation is
\[
\{(p-1)+(p-2)(p+p^2+\cdots+p^{h-1})+p(1+p+\cdots+p^{h-2})\}N=(p^{h}-1)N.
\]
The coefficients of $B_i$, $C_i$, and $D_{i,j}$ in $G$ are denoted by $a_{B_i}$, $a_{C_i}$ and $a_{D_{i,j}}$, respectively.
Furthermore, we define $a_B:=a_{B_1}\cdots a_{B_{p-1}}$, $a_C:=a_{C_1}\cdots a_{C_{p-2}}$ and $a_{D_i}:=a_{D_{i,1}}\cdots a_{D_{i,p}}$.
If we prove that the coefficient $\alpha \in k$ of the monomial
\[
a_B B \cdot (a_C C)^{p+p^2+\cdots+p^{h-1}} \cdot (a_{D_1} D_1) \cdot (a_{D_2} D_2)^p \cdots (a_{D_{h-1}} D_{h-1})^{p^{h-2}}
\]
appearing in $G_h$ is non-zero, then we obtain the desired result, that is, $P'_h$ is non-empty.
\begin{claim}
The above coefficient $\alpha$ is non-zero.
\end{claim}
\begin{claimproof}
First, we note that 
\[
\widetilde{\Delta}_1(G^{p-1})\equiv -G^{p(p-2)}\widetilde{\Delta}_1(G)
\]
modulo the image of the Frobenius morphism,
and thus $G_h$ can be denoted by
\[
G_h \equiv \pm G^{p-1} \cdot (G^{p-2})^{p+p^2+\cdots+p^{h-1}}\cdot \widetilde{\Delta}_1(G)^{1+p+\cdots+p^{h-2}}
\]
modulo the Frobenius image.
In particular, $\alpha$ is the coefficient of
\[
a_B B \cdot (a_C C)^{p+p^2+\cdots+p^{h-1}} \cdot (a_{D_1} D_1) \cdot (a_{D_2} D_2)^p \cdot \ldots \cdot (a_{D_{h-1}} D_{h-1})^{p^{h-2}}
\]
in
\[
G^{p-1} \cdot (G^{p-2})^{p+p^2+\cdots+p^{h-1}}\cdot \widetilde{\Delta}_1(G)^{1+p+\cdots+p^{h-2}}
\]
up to sign.
If we denote the coefficient of $(a_C C)^p \cdot a_{D_{h-1}} D_{h-1}$ in $G^{p(p-2)}\widetilde{\Delta}_1(G)$ by $\beta_1$ and the coefficient of
\[
a_B B \cdot (a_C C)^{p+p^2+\cdots+p^{h-2}} \cdot (a_{D_1} D_1) \cdot (a_{D_2} D_2)^p \cdot \ldots \cdot (a_{D_{h-2}} D_{h-2})^{p^{h-3}}
\]
in $G_{h-1}$ is denoted by $\alpha_1$, 
then we have $\alpha_1 \beta_1^{p^{h-2}}=\alpha$ up to sign.
By computation, $\beta_1=(-(p-2)!)^p (p-1)!=1$.
Repeating such arguments, $\alpha$ coincides with the coefficient of $(a_B B)(a_C C)^p (a_{D_1} D_1)$ in $G^{p^2-1}\widetilde{\Delta}_1(G)$ up to sign, which is computed as
\[
((p-1)!)^p(p-1)! \binom{2p-1}{p-1}=1.
\]
Therefore, $\alpha$ is $\pm 1$, as desired.
\end{claimproof}
\end{proof}

\begin{rmk}
By construction, the dimension of $X_h$ is $p^h-3$.
In \cite[Example~6.7]{KTY1}, we construct a $(2^h-1)$-dimensional Calabi--Yau hypersurface over $\F_2$ of height $2h$.
\end{rmk}

\subsection{Bielliptic surfaces}
In this subsection, we study the quasi-$F$-split height of a bielliptic surface over an algebraically closed field $k$ of characteristic $p>0$. 
\begin{rmk}[\cite{Mehta-Srinivas}]
Let $\pi \colon Y \to X$ be a finite \'etale morphism of normal varieties with $K_Y \sim_{\Q} 0$. In this case, we can see that $X$ is $F$-split if and only if so is $Y$.
Indeed, if $X$ is $F$-split, then the pullback of a splitting section gives a splitting section of $Y$, which shows $Y$ is $F$-split.
Conversely, we assume that $Y$ is $F$-split.
Passing to a Galois closure and using the previous argument, we may assume $\pi$ is a Galois quotient.
The Galois group is denoted by $G$.
Since we have
\[
\Hom_Y(F_*\sO_Y,\sO_Y) \cong H^0(Y, (1-p)K_Y) \cong k,
\]
the unique splitting section, which sends $F_*1$ to $1$, is $G$-equivariant.
Therefore, the section descends to $X$ and it is also a splitting section.
In particular, if $Y$ is a fiber product of elliptic curves $E_1 \times E_0$ and $X$ is a bielliptic surface, then it follows that the $F$-splitting of $X$ is equivalent to the $F$-splitting of both $E_1$ and $E_0$. 

Next, let us consider the case where the quasi-$F$-split height of $Y = E_1 \times E_0$ is $2$.
Note that $\sht (Y) \leq n$ is equivalent to the splitting of the morphism $\sO_{Y} \rightarrow F_* \var{W}_{n} \sO_{Y}$ (see \cite[Proposition 2.8]{KTY1}).
In this subsection, we refer to any morphisms $F_* \var{W}_n \sO_Y \to \sO_Y$ as well as $F_* W_n \sO_Y \to \sO_Y$ that give such a splitting as a \emph{splitting sections}.
Recall that there is an exact sequence
\[
0 \rightarrow F_* ( F_*\sO_Y/\sO_Y) \rightarrow F_* \var{W}_2 \sO_Y \rightarrow F_* \sO_Y \rightarrow 0
\]
(see \cite[Proposition 2.10]{KTY1}).
Since $Y$ is not $F$-split, we have
\[
\dim_k \Hom_Y (F_* \sO_Y / \sO_Y, \sO_Y) =1.
\]
Therefore, $\dim_{k}\Hom_Y( F_*\var{W}_2 \sO_{Y},\sO_Y) \leq2$. 
Moreover, since 
\[
\Hom (F_* \sO_Y, \sO_Y) \hookrightarrow \Hom (F_* \var{W}_{2}\sO_{Y}, \sO_Y)
\]
is not an isomorphism since $\sht (Y) =2$, we have $\dim_{k}\Hom_Y( F_*\var{W}_2 \sO_{Y},\sO_Y) = 2$.
Hence, it is not clear whether a splitting section is $G$-equivariant.
In Theorem \ref{thm:bielliptic}, we show that, for a certain case with $p=2$, the quasi-$F$-split heights of $X$ and $Y$ coincide by verifying the $G$-invariance of a splitting section.
On the other hand, in another case with $p=2$ or $p=3$, 
we produce a non-$G$-equivariant splitting section; in those cases, $X$ is not quasi-$F$-split even though $Y$ is.
\end{rmk}

First, we recall the definition of bielliptic and quasi-bielliptic surfaces.
\begin{defn}
Let $k$ be an algebraically closed field, and $X$ a smooth projective surface over $k$ with $K_X \sim_{\Q} 0$.
\begin{enumerate}
    \item 
    We say $X$ is a \emph{bielliptic surface} over $k$ if the Albanese morphism of $X$ is an elliptic fibration over an elliptic curve.
    \item
    We say $X$ is a \emph{quasi-bielliptic surface} over $k$ if the Albanese morphism of $X$ is a quasi-elliptic fibration over an elliptic curve. Note that quasi-bielliptic surfaces exist only in characteristic two and three.
\end{enumerate}
\end{defn}

In \cite[Section~3]{Bombieri2}, bielliptic surfaces are classified as follows.
A bielliptic surface $X$ can be written as a quotient
\[
X = (E_1 \times E_0)/G,
\]
where $E_0$ and $E_1$ are elliptic curves,
and $G \subset E_1$ is a finite subgroup scheme acting faithfully on $E_0$,
such that $G$ and the action of $G$ on $E_0$ are as in one of the following cases
(listed by the action of standard generators of $G$):
\begin{itemize}
\item[(a1)]
$G \simeq \Z/2\Z$, acting on $E_0$ by $y \mapsto -y$.
\item[(a2)]
$G \simeq (\Z/2\Z)^2$, acting on $E_0$ by
$y \mapsto -y$ and $y \mapsto y+c$ for some $c \in E_0[2]$. We have $p\neq 2$ in this case.
\item[(a3)]
$G \simeq (\Z/2\Z) \times \mu_2$, acting on $E_0$ by
$y \mapsto -y$, where $\mu_2$ acts on $E_0$ by translations.
\item[(b1)]
$G \simeq \Z/3\Z$, acting on $E_0$ by $y \mapsto \omega y$,
where $j(E_0) = 0$ and $\omega \in \Aut_0(E_0)$ has order~$3$. 
\item[(b2)]
$G \simeq \Z/3\Z$, acting on $E_0$ by $y \mapsto \omega y$
and $y \mapsto y+c$, where $\omega$ is as above and
$c \in E_0[3]$ satisfies $\omega(c) = c$.
We have $p \neq 3$ in this case.
\item[(c1)]
$G \simeq \Z/4\Z$, acting on $E_0$ by $y \mapsto i y$,
where $j(E_0) = 12^3$ and $i \in \Aut_0(E_0)$ has order~$4$.
\item[(c2)]
$G \simeq \Z/4\Z \times \Z/2\Z$, acting on $E_0$ by
$y \mapsto i y$ and $y \mapsto y+c$,
where $i$ is as above and $c \in E_0[4]$ satisfies $i(c) = c$. We have $p\neq 2$ in this case.
\item[(d)]
$G \simeq \Z/6\Z$, acting on $E_0$ by $y \mapsto -\omega y$,
where $E_0$ and $\omega$ are as in~(b).
\end{itemize}

\begin{thm}\label{thm:bielliptic}
\begin{enumerate}
    \item[\textup{(1)}] 
Let $X$ be a quasi-bielliptic surface over an algebraically closed field $k$. Then $X$ is not quasi-$F$-split.
\item[\textup{(2)}]
Let $X$ be a bielliptic surface over an algebraically closed field $k$ of characteristic $p>0$.
As explained above, $X$ can be expressed as a quotient
\[
X = (E_1 \times E_0)/G
\]
in one of the cases \textup{(a1)}--\textup{(d)}.
Then we have
\[
\sht(X)=
\begin{cases}
\infty & \text{if $p=2$ in case \textup{(c1)} or $p=3$ in case \textup{(b1)} or \textup{(d)}},\\[4pt]
\sht(E_{1}\times E_{0}) & \text{otherwise.}
\end{cases}
\]
Note that, in the former case, we have $\sht (E_1 \times E_0) =2$.
\end{enumerate}
\end{thm}

\begin{proof}
The assertion (1) follows from Corollary \ref{cor:genus one gibration surface}.

We prove the assertion (2).
In the following, we put $Y:= E_1 \times E_0$.
First, if the order of $G$ is coprime to $p$,
then we have $\sht(X) = \sht (Y)$ by Propositions \ref{prop:fiber space} and \ref{prop:etale extension}.
In particular, we have $\sht(X) = \sht(Y)$ when $p > 3$, 
when $p = 2$ in cases \textup{(b1)} and \textup{(b2)}, 
and when $p = 3$ in cases \textup{(c1)} and \textup{(c2)}.

Next, if the elliptic curves $E_{1}$ and $E_{0}$ are ordinary, we have $\sht (X) =1$ by \cite[Proposition 7.2]{Ejiri}.
When $p = 2$ in case \textup{(a3)}, the elliptic curves $E_1$ and $E_0$ are ordinary 
(note that if $E_0$ admits a subgroup scheme $\mu_p$, then $E_0$ is ordinary).

Therefore, it suffices to treat the remaining cases below.
\begin{itemize}
    \item[(i)] 
    $p=2$ in case (a1)
    \item[(ii)]
    $p=2$ in case (d)
    \item[(iii)]
    $p=2$ in case (c1)
    \item[(iv)]
    $p=3$ in case (b1)
    \item[(v)]
    $p=3$ in case $(d)$
\end{itemize}
Note that $E_1$ has non-trivial $p$-torsion in every case, so $E_1$ is ordinary.
As stated before, we may assume that $E_0$ is supersingular.
Note that, in this case, $Y=E_1 \times E_0$ is of height $2$ by Proposition \ref{prop:fiber products}.
We fix an embedding $E_ 1 \subset \P^2_k$ as a hypersurface, and let $Q \in k[x,y,z]$ be the equation.
Moreover, we also fix $P \in k[x,y,z]$
which defines an elliptic curve isomorphic to $E_{0}$ (note that there is only one isomorphism class of supersingular elliptic curves in characteristic $2$ or $3$).
The equation $P$ is essentially unique, but for convenience, we shall replace $P$ as needed in the following discussion.
In the following, we regard $E_{1}$ and $E_{0}$ as closed subschemes of $\P^{2}_{k}$ by these equations.
We put $S_{E_{1}}:= k[x,y,z]/Q$ and $S_{E_{0}}:=k[X,Y,Z]/P$.
We also put 
\[
S:= S_{E_{1}} \otimes_{k} S_{E_{0}} \simeq k[x,y,z,X,Y,Z]/(Q,P).
\]
As in Convention \ref{conv:graded ring}, we define $u$ and $\Delta_1$ for this $S$ by the $p$-basis
\[
\{
F_* (x^{i_1} y^{i_2} z^{i_3} X^{i_4} Y^{i_5} Z^{i_6}) \mid 0 \leq i_1, \ldots, i_6 \leq p-1
\}.
\]
Similarly, we also define $u_{E_j}$, $\Delta_1^{{E_j}}$ for $S_{E_j}$ $(j=0, 1).$
We put
\[
\theta_{E_0} (F_* a ) = u_{E_0} (F_{*} (a \Delta_{1}^{E_0} (P))), \quad 
\theta_{{E_1}} (F_* b ) = u_{{E_1}} (F_{*} (b \Delta_{1}^{{E_1}} (Q)))
\]
for $a \in S_{E_0}$ and $b \in S_{E_1}$.
Note that the action of $G$ on $E_j$ induces an action on $S_{E_j}$ for $j=0,1$, and hence also on $S$.

First, we will treat the case (i).
In this case, we prove $\sht (X) =2$.
We may take
\[
P:= X^3+Y^3+Z^3 
\]
in this case.
If we fix an origin of $E_{0}$ as $(1:1:0)$, then the map $[-1]$ can be written as
\[
[-1] \colon (X:Y:Z) \mapsto (Y:X:Z).
\]
We will show that $X$ is of height $2$.
Recall that we have a splitting section
\[
\psi_{f_{1},f_{2}}
\colon 
F_* \var{W}_{2} S
\rightarrow S.
\]
Here, $f_{1}$ and $f_{2}$ are defined by
\begin{equation}
\label{eqn:f1f2}
f_1 = g_1 \otimes u_{{E_1}} (F_* h)  \quad \textup{and} \quad f_2 = g_2 \otimes h,
\end{equation}
where $g_{1}:=\theta_{{E_0}}(F_*P^{p-1})$, $g_{2}:=P^{p^2-1}$, and $h:= Q^{p^2-1}$ (see the proof of \cite[Lemma~4.3, Theorem~4.8, Lemma~4.10, Theorem~4.11, and Theorem~5.8]{KTY1} and Proposition \ref{prop:fiber products}), and $\psi_{f_1, f_2}$ is an $S$-module morphism satisfying
\begin{eqnarray*}
\psi_{f_1, f_2}(F_* ([a])) &=& u (F_* (f_{1} a)) + u^2 (F^2_* (f_2 \Delta_1 (a))), \\
\psi_{f_{1}, f_2} (F_* (V[a])) &=& u^2 (F_*^2 (f_2 a) )
\end{eqnarray*}
for any $a\in S$ and $s \in \{0,1 \}$.
First, we show that $\psi_{f_1, f_2}$ is $G$-equivariant, i.e.,\ for the generator $g  \in \Z/2\Z$, we show
\[
\psi_{f_1,f_2} \circ g = g \circ \psi_{f_1,f_2},
\]
where we use the natural $G$-action induced on $F_{*}\overline{W}_{2}(S)$.
It suffices to show the equalities 
\begin{equation}
\label{eqn:firstpsi}
\psi_{f_1,f_2} \circ g (F_* ([x^{i_1} y^{i_2} z^{i_3} X^{i_4} Y^{i_5} Z^{i_6}]
))= g \circ \psi_{f_1,f_2} (F_*([x^{i_1} y^{i_2} z^{i_3} X^{i_4} Y^{i_5} Z^{i_6}])),
\end{equation}
\begin{equation}
\label{eqn:secondpsi}
\psi_{f_1,f_2} \circ g (F_* (V[x^{i_1} y^{i_2} z^{i_3} X^{i_4} Y^{i_5} Z^{i_6}]
))= g \circ \psi_{f_1,f_2} (F_*(V[x^{i_1} y^{i_2} z^{i_3} X^{i_4} Y^{i_5} Z^{i_6}]))
\end{equation}
for $i_1, \ldots, i_6 \in \Z_{\geq0}.$
Let $g_x, g_y, g_z \in k[x,y,z]$ be representatives of $g (x), g(y), g(z) \in k[x,y,z]/Q$.
First, we show (\ref{eqn:secondpsi}).
The left-hand side of (\ref{eqn:secondpsi}) is equal to 
\begin{eqnarray}
&u^2 (F^2_{*} (P^{p^2-1} Q^{p^2-1} g_x^{i_1} g_y^{i_2} g_z^{i_3} Y^{i_4} X^{i_5} Z^{i_6})) \\
= & u_{E_0}^2 (F^2_*P^{p^2-1} Y^{i_4} X^{i_5} Z^{i_6}) u_{E_1}^2 (F^2_*Q^{p^2-1} g_x^{i_1} g_y^{i_2} g_z^{i_3} ).
\label{eqn:secondprime}
\end{eqnarray}
Since $E_1$ is an ordinary elliptic curve,
the $F$-splitting section
\begin{equation}
\label{eqn:G-invariantsplitting}
u_{E_1} ( F_* (Q^{p-1} -)) \colon F_* S_{E_1} \rightarrow S_{E_1}
\end{equation}
is $G$-equivariant. Therefore, we have
\begin{eqnarray*}
u_{E_1}^2 (F^2_* (Q^{p^2-1} g_x^{i_1} g_y^{i_2} g_z^{i_3} )) &=& u_{E_1} (F_* (Q^{p-1} u_{E_1} (F_* (Q^{p-1} g (x^{i_1} y^{i_2} z^{i_3}))))) \\
&=& g( u_{E_1}(F_* (Q^{p-1} u_{E_1} (F_* (Q^{p-1} x^{i_1} y^{i_2} z^{i_3})) )))\\
&=& g ( u_{E_1}^2 (F^2_* (Q^{p^2-1} x^{i_1} y^{i_2} z^{i_3}))) \in S_{E_1} \subset S.
\end{eqnarray*}
On the other hand, 
we have
\[
u_{E_0}^2 (F^2_*P^{p^2-1} Y^{i_4} X^{i_5} Z^{i_6}) = g (u_{E_0}^2 (F_*^2 (X^{i_4} Y^{i_5} Z^{i_6})))
\]
since $P$ is a symmetric polynomial and $u_{E_0}$ is clearly $G$-equivariant.
Therefore, (\ref{eqn:secondprime}) is equal to 
\begin{eqnarray*}
&&g (u_{E_0}^2 (F_*^2 (X^{i_4} Y^{i_5} Z^{i_6})) 
 u_{E_1}^2 (F^2_* (Q^{p^2-1} x^{i_1} y^{i_2} z^{i_3}))) \\
 &=&  g \circ \psi_{f_1,f_2} (F_*(V[x^{i_1} y^{i_2} z^{i_3} X^{i_4} Y^{i_5} Z^{i_6}]))
\end{eqnarray*}
as desired.
Next, we show (\ref{eqn:firstpsi}).
Note that a direct computation shows that $g_1 = XYZ$ and $u_{E_1} (Q^{p^2-1}) = Q^{p-1} u_{E_1} (Q^{p-1}) = c Q^{p-1}$, where $c\in k^{\times}$.
Therefore, the left-hand side of (\ref{eqn:firstpsi}) is equal to
\begin{eqnarray*}
u(F_* (cXYZQ^{p-1} g_x^{i_1} g_y^{i_2} g_z^{i_3} Y^{i_4} X^{i_5} Z^{i_6})) + u^2 (F^2_* (P^{p^2-1} Q^{p^2-1} \Delta_1 (g_x^{i_1} g_y^{i_2} g_z^{i_3} Y^{i_4} X^{i_5} Z^{i_6}))).
\end{eqnarray*}
The second term is equal to
\begin{eqnarray*}
&&u (F_* (P^{p-1} Q^{p-1} u( F_*( P^{p-1} Q^{p-1} \Delta_1 (g_x^{i_1} g_y^{i_2} g_z^{i_3} Y^{i_4} X^{i_5} Z^{i_6})))))\\
&=&
u (F_* (P^{p-1} Q^{p-1} u( F_*( P^{p-1} Q^{p-1} Y^{2i_4} X^{2i_5} Z^{2i_6} \Delta_1 (g_x^{i_1} g_y^{i_2} g_z^{i_3}))))) \\
&=& 
u (F_* (P^{p-1} Q^{p-1} Y^{i_4} X^{i_5} Z^{i_6} u( F_*( P^{p-1} Q^{p-1} \Delta_1 (g_x^{i_1} g_y^{i_2} g_z^{i_3})))))=0,
\end{eqnarray*}
where the last equality holds since $u_{E_0}(F_* P^{p-1}) =0$.
Therefore, by using $G$-equivariance of (\ref{eqn:G-invariantsplitting}) again, we can show that (\ref{eqn:firstpsi}) equals 
\begin{eqnarray*}
&&u_{E_0}(F_*  (X^{i_5 +1}Y^{i_4 +1}Z^{i_6 +1})) u_{E_1} (F_* (cQ^{p-1} g_x^{i_1} g_y^{i_2} g_z^{i_3}))\\
&=& g (u_{E_{0}} (F_* (X^{i_4 +1} Y^{i_5 +1} Z^{i_6 +1})) ) g( u_{E_1} (F_* (cQ^{p-1} x^{i_1} y^{i_2} z^{i_3}))) \\
&=& g (u (F_* (cXYZ Q^{p-1} x^{i_1} y^{i_2} z^{i_3} X^{i_4} Y^{i_5} Z^{i_6})))\\
&=& g \circ \psi_{f_1,f_2} (F_*([x^{i_1} y^{i_2} z^{i_3} X^{i_4} Y^{i_5} Z^{i_6}])),
\end{eqnarray*}
as desired.
In the last equality, we use
\[
u^2 (F^2_* (P^{p^2-1} Q^{p^2-1} \Delta_1 (x^{i_1} y^{i_2} z^{i_3} X^{i_4} Y^{i_5} Z^{i_6}))) =0.
\]
Next, we show that $\psi_{f_1, f_2}$ induces a splitting section on $X$.
Note that $E_0 \subset \P^2$ is covered by $G$-stable affine open subsets
\[
\{ D_+ (a_j)  \}_{j \in J_0} 
\]
for some homogeneous elements $a_j \in S_{E_0}$ (cf.\ the proof of \cite[II, 7, Theorem]{Mumford} for example).
Similarly, $E_1 \subset \P^2$ is covered by $G$-stable affine open subsets 
\[
\{ D_+ (b_k)  \}_{k \in J_1} 
\]
for some homogeneous elements $b_k \in S_{E_1}$.
Therefore, $E_1 \times E_0 \subset \P^2 \times \P^2$ is covered by $G$-stable affine open subsets 
\[
\{
D_+ (a_j b_k)
\}_{j\in J_0, k\in J_1}.
\]
We put
$S_{j,k} := S[1/{(a_jb_k)}]_{0}.$
Note that we have
\[
D_+ (a_j b_k) \simeq \Spec S_{j,k}
\]
By the proof of Proposition \ref{wprop:multigraded algebra and proj}, $\psi_{f_1, f_2}$ induces a morphism
\[
\psi_{f_1, f_2}^{(j,k)} \colon F_* W_2 (S_{j,k})  \rightarrow S_{j,k}
\]
which in turn gives a splitting section
\[
F_* W_2 \sO_{Y} \rightarrow \sO_{Y}
\]
by gluing.
By the choice of $a_j$ and $b_k$, there is a natural $G$-action on $S_{j,k}$.
Moreover, the $G$-quotient $E_1 \times E_0 \rightarrow X$ restricts $\Spec S_{j,k} \rightarrow \Spec S_{j,k}^{(G)}$, where $S_{j,k}^{(G)}$ is the $G$-invariant part of $S_{j,k}$.
Moreover, since $\psi_{f_1, f_2}$ is $G$-equivariant, so is $\psi_{f_1.f_2}^{(j,k)}$.
Therefore, $\psi_{f_1,f_2}^{(j,k)}$ induces a morphism
\[
\overline{\psi_{f_1, f_2}^{(j,k)}} \colon F_* W_2 (S_{j,k}^{(G)}) \rightarrow S_{j,k}^{(G)},
\]
which gives a desired splitting section
\[
\psi \colon F_*W_2 \sO_{X} \rightarrow \sO_X
\]
by gluing.
This shows that $X$ is $2$-quasi-$F$-split.
By Proposition \ref{prop:etale extension}, $X$ is not $F$-split. Therefore, we have $\sht (X) =2$ as desired.

Next, we treat the case (ii).
In this case, we prove $\sht (X) =2$.
Note that the $G$-quotient morphism can be factored as
\[
E_{1} \times E_{0} \rightarrow (E_{1} \times E_{0})/ (\Z/2\Z) \rightarrow X,
\]
where the second morphism is a finite \'{e}tale morphism of degree $3$.
By (i), we have $\sht ( E_1 \times E_0 / (\Z/2\Z)) = 2$.
Therefore, by Proposition \ref{prop:fiber space} and Proposition \ref{prop:etale extension}, we have $\sht (X) = 2$ as desired.

Next, we treat the case (iii).
In this case, we prove $\sht (X) = \infty$.
We may take 
\[
P:= X^3 + Y^2Z + YZ^2.
\]
If we fix an origin of $E_{0}$ as $(0:1:0)$, then automorphisms of order 4 are the following six (cf.\ \cite[Appendix A, the proof of Proposition 1.2]{Silverman}):
\[
(X:Y:Z) \mapsto (X + s^2 Z : Y + sX + tZ :Z),
\]
where $s, t \in k$ are solutions of
\[
s^3 =1, \quad  1 + t + t^2 =0.
\]
We fix $s$ and $t$ corresponding to the automorphism $i$.
Suppose, for a contradiction, that 
$n := \sht (X) < \infty$.
By the proof of Proposition \ref{prop:etale extension}, there is a $G$-equivariant splitting section
\[
\varphi F_{*} \var{W}_n \sO_{Y} \rightarrow \sO_{Y}.
\]

We show that there is a $G$-equivariant  splitting section
\[
F_{*} \var{W}_2 \sO_{Y} \rightarrow \sO_{Y}.
\]
Suppose, for a contradiction, that there is no such a splitting section. We have $n\geq 3.$
We define
\[
\varphi_2' \colon 
F^{n-1}_*\var{W}_2\sO_Y 
\xrightarrow{V^{\,n-2}} 
F_* \var{W}_n\sO_Y 
\xrightarrow{\varphi} 
\sO_Y.
\]
Then $\varphi_2'$ is $G$-equivariant.
Fixing an isomorphism $\sO_Y \simeq \omega_{Y}$, we have
\[
\begin{aligned}
\Hom(F^{n-1}_*\var{W}_2\sO_Y,\sO_Y)
&\simeq 
\Hom(F^{n-1}_*\var{W}_2\sO_Y,\omega_Y)\\
&\simeq 
F^{n-2}_*\Hom(F_*\var{W}_2\sO_Y,\omega_Y)\\
&\simeq 
F^{n-2}_*\Hom(F_*\var{W}_2\sO_Y,\sO_Y).
\end{aligned}
\]
Let $\varphi_2 \colon F_*\var{W}_2\sO_Y \to \sO_Y$ be a homomorphism corresponding to $\varphi_2'$ via the above isomorphism.
We have
$
\varphi'_2 = \mathrm{Tr}^{n-2} \circ F^{n-2}_*  (\varphi_2), 
$
where \[
\mathrm{Tr} \colon F_* \sO_Y \simeq F_*\omega_Y \rightarrow \omega_Y \simeq \sO_Y
\]
is the trace map.
Note that $\Hom (F_*\sO_Y, \sO_{Y}) \simeq k$ is a trivial $G \simeq \Z/2\Z$-representation since $\chara k =2$. Therefore, $\varphi_2$ is $G$-equivariant.
Note that $\Hom(F_*\var{W}_2\sO_Y,\sO_Y)$ is spanned by $\Hom (F_* \sO_Y, \sO_Y) \simeq k$ and a (single) splitting section, which is not $G$-equivariant by the assumption for contradiction (cf.\ the discussion at the beginning of this subsection). Note that any element of $\Hom(F_*\sO_Y, \sO_Y)$ maps $ F_*1$ to $0$ since $Y$ is not $F$-split. 
Therefore, $\varphi_2$ is the image of the dual map
\[
\Hom(F_*\sO_Y,\sO_Y) 
\longrightarrow 
\Hom(F_*\var{W}_2\sO_Y,\sO_Y)
\]
of the restriction morphism $F_*\var{W}_2\sO_Y \to F_*\sO_Y$.
From the exact sequence 
\[
0 \to F_*B_{Y,1}:=F^2_*\sO_Y/\mathrm{Im}(F)
\xrightarrow{V} F_*\var{W}_2\sO_Y 
\to F_*\sO_Y \to 0,
\]
the composition
\[
F_*B_{Y,1} \xrightarrow{V} F_*\var{W}_2\sO_Y 
\xrightarrow{\varphi_2} \sO_Y
\]
is zero.
Therefore, the composition
\[
F^{n-1}_*B_{Y,1} 
\xrightarrow{V} F^{n-1}_*\var{W}_2\sO_Y 
\xrightarrow{V^{\,n-2}} F_*\var{W}_n\sO_Y 
\xrightarrow{\varphi} \sO_Y
\]
is also zero.
By the exact sequence
\[
0 \to F^{n-1}_*B_{Y,1} 
\xrightarrow{V^{\,n-1}} F_*\var{W}_n\sO_Y 
\xrightarrow{R} F_*\var{W}_{n-1}\sO_Y \to 0,
\]
the homomorphism $\varphi$ induces a $G$-equivariant splitting section
\[
F_*\var{W}_{n-1}\sO_Y \to \sO_Y.
\]
Repeating this argument inductively, we eventually obtain a $G$-equivariant splitting section
\[
F_*\var{W}_2\sO_Y \to \sO_Y,
\]
which contradicts our assumption.
Hence, we conclude that 
there exists a $G$-equivariant section
\[
s_1 \colon F_* \var{W}_2 \sO_{Y} \rightarrow \sO_{Y}. 
\]
As we saw before (cf.\ the discussion at the beginning of this subsection), $\Hom (F_* \var{W}_2 \sO_{Y}, \sO_{Y})$ is spanned by $s_1$ and an image $s_2$ of a generator of $\Hom (F_* \sO_{Y}, \sO_{Y}) \simeq k$ via the natural injection
\[
\Hom (F_* \sO_{Y}, \sO_{Y}) \rightarrow 
\Hom (F_* \var{W}_2 \sO_{Y}, \sO_{Y}).
\]
Since this injection is $G$-equivariant, $s_2$ is also $G$-equivariant.
By the proof of Proposition \ref{wprop:multigraded algebra and proj}, 
there is a $G$-equivariant multi-graded splitting section corresponding to $s_1$
\[
\widetilde{s}_1 \colon F_* \var{W}_2 S \rightarrow S
\]
which is $G$-equivariant.
By replacing $s_1$ with $s_1 + rs_2$ for some $r \in k$, we may assume that $\widetilde{s}_1$ is equal to the morphism 
\[
\psi_{f_1, f_2} \colon F_* \var{W}_2 S \rightarrow S,
\]
where $f_1$ and $f_2$ are defined by (\ref{eqn:f1f2}).
In our case, we have
\[
f_1 = XYZc Q^{p-1}, \quad f_2 =P^{p^2-1}Q^{p^2-1},
\]
where we put $u_{E_1}(F_* Q^{p-1}) =c \in k^{\times}$ as before.
To derive a contradiction, it suffices to show that $\psi_{f_1,f_2}$ is not $G$-equivariant.
Let $g \in G \simeq \Z/4\Z$ be a generator.
We show that 
\begin{equation}
\label{eqn:contradictionc1}
g(\psi_{f_1,f_2} (F_* [X^3 Y^2 Z]) ) \neq  \psi_{f_1,f_2}( g (F_* [X^3 Y^2 Z])).
\end{equation}
Since
\begin{eqnarray*}
\psi_{f_1,f_2} (F_* [X^3 Y^2 Z]) &=& 
u (F_* (f_1 X^3 Y^2 Z)) + u^2 (F^2_* f_2 \Delta_1 (X^3Y^2Z)) =0,
\end{eqnarray*}
the left-hand side of (\ref{eqn:contradictionc1}) equals $0$.
On the other hand, the right-hand side of (\ref{eqn:contradictionc1}) is equal to 
\begin{eqnarray*}
&&\psi_{f_1,f_2} (F_* [(X+s^2Z)^3 (Y+sX+tZ)^2 Z]) \\
&=& u(F_* ( f_1 (X+s^2Z)^3 (Y+sX+tZ)^2 Z)) + u^2(F^2_* (f_2 \Delta_1 ((X+s^2Z)^3 (Y+sX+tZ)^2 Z))).
\end{eqnarray*}
Note that we have
\begin{equation}
\label{eqn:321}
\begin{aligned}
&(X+s^2Z)^3 (Y+sX+tZ)^2 Z \\
=\ & Z (s^2 X^5 +  sX^4Z + X^3Y^2 + t X^3Z^2 + s^2 X^2Y^2Z \\
& +  s^2t X^2 Z^3 + sXY^2Z^2 + st^2 XZ^4 + Y^2Z^3 + t^2Z^5).
\end{aligned}
\end{equation}
Therefore, we have
\begin{eqnarray*}
&&u(F_* (f_1 (X+s^2Z)^3 (Y+sX+tZ)^2 Z)) \\
&=& u(F_* (XYZcQ^{p-1}(X+s^2Z)^3 (Y+sX+tZ)^2 Z )) \\
&=& c^{1+ 1/p} u_{E_0} (F_* (XYZ^2 (sX^4Z +s^2 X^2 Y^2 Z + s^2t X^2 Z^3 + Y^2Z^3 + t^2Z^5))\\
&=& c^{1+ 1/p} (s^{1/2} X^2Z + s XYZ  + st^{1/2}XZ^2 + YZ^2 + t Z^3). 
\end{eqnarray*}
On the other hand, noting that no monomial with an odd exponent of $y$ appears in (\ref{eqn:321}), a direct computation shows that 
\begin{eqnarray*}
&&u_{E_0} (F_*(P^{p-1}\Delta_1 ((X+s^2Z)^3 (Y+sX+tZ)^2 Z) )) \\
&=& u_{E_0} (F_*(YZ^2\Delta_1 ((X+s^2Z)^3 (Y+sX+tZ)^2 Z) )) \\
&=& X^4Z^2  + s X^2 Y^2 Z^2 + st X^2 Z^4 + s^{1/2} Y^2Z^4 + s^{1/2} t^{1/2} Z^6. 
\end{eqnarray*}
Therefore, we have
\begin{eqnarray*}
&& u^2(F^2_* (f_2 \Delta_1 ((X+s^2Z)^3 (Y+sX+tZ)^2 Z))) \\
&=& u_{E_1}^2 (F^2_* Q^{p^2-1}) u_{E_0}^2 (F^2_* (P^{p^2-1}\Delta_1 ((X+s^2Z)^3 (Y+sX+tZ)^2 Z))) \\
&=& c^{1+ 1/p} u_{E_0} (F_* (P^{p-1} u_{E_0} (F_*(P^{p-1}\Delta_1 ((X+s^2Z)^3 (Y+sX+tZ)^2 Z) )))) = 0.\\
\end{eqnarray*}
Consequently, the right-hand-side of (\ref{eqn:contradictionc1}) is equal to
\[
c^{1+ 1/p} (s^{1/2} X^2Z + s XYZ  + st^{1/2}XZ^2 + YZ^2 + t Z^3) \in S,
\]
which is not equal to zero (since it is non-zero at $(0:1:1) \in E_0$) and this yields a contradiction.

Next, we treat the case (iv).
We may take
\[
P := Y^2Z-X^3-XZ^2
\]
in this case.
If we fix an origin of $E_0$ as $(0:1:0)$, then automorphisms of order 3 are the following two (cf.\ \cite[Appendix A, the proof of Proposition 1.2]{Silverman}):
\[
(X:Y:Z) \mapsto (X \pm iZ:Y:Z),
\]
where $i$ is a fixed square root of $-1\in k$.
Suppose that $n:=\sht(X)<\infty$.
By the same argument as in case~(iii), it reduces to showing that the morphism $\psi_{f_{1},f_{2}}$ 
is not $G$-equivariant.
Here, $f_{1}$ and $f_{2}$ are defined by (\ref{eqn:f1f2}).
In our case, 
we have
\[
f_1 =(X^5Z-X^2Y^2Z^2 +X^3Z^3)cQ^{p-1}, \quad f_2 = P^{p^2-1} Q^{p^2-1},
\]
where we put $u_{E_1} (F_* Q^{p-1}) =c \in k^{\times}.$
Let $g\in G \simeq \Z/3\Z$ be a generator.
We show that
\begin{equation}
\label{eqn:contradictionb1}
g(\psi_{f_1,f_2} (F_* [Y^2 Z]) ) \neq  \psi_{f_1,f_2}( g (F_* [Y^2 Z])).
\end{equation}
Since
\begin{eqnarray*}
\psi_{f_1,f_2} (F_* [Y^2 Z]) &=& u(F_* (f_1 Y^2Z)) + u^2 (F^2_* (f_2 \Delta_1 (Y^2Z))) \\
&=& c^{1+ 1/p} u_{E_0} (F_* (Y^2 Z(X^5Z-X^2Y^2Z^2 + X^3Z^3))) \\
&=& c^{1+1/p} X,
\end{eqnarray*}
the left-hand side of (\ref{eqn:contradictionb1}) equals
\[
g(c^{1+1/p} X) = c^{1+1/p} ( X \pm iZ).
\]
On the other hand, since 
\[
g(F_* [Y^2Z]) = F_* [Y^2Z],
\]
the right-hand side of (\ref{eqn:contradictionb1}) equals
$c^{1+/1/p} X$.
Since $\pm c^{1+1/p} iZ \in S$ is non-zero, this yields a contradiction.

Finally, we treat the case (v).
The $G$-quotient morphism can be factored as
\[
E_1 \times E_0 \rightarrow (E_1 \times E_0)/ (\Z/3\Z) \rightarrow X,
\]
where the second morphism is a finite \'{e}tale morphism of degree 2.
By (iv), we have $\sht ((E_{1} \times E_0) / (\Z/3\Z)) = \infty.$
Therefore, by Proposition \ref{prop:etale extension}, we have
$\sht (X) = \infty$ as desired.
\end{proof}

\subsection{Fano varieties}
In this subsection, we compute quasi-$F$-split heights of certain Fano varieties.
In particular, we show that in dimension greater than two, there always exist non-quasi-$F$-split smooth Fano varieties, while all smooth del Pezzo surfaces are quasi-$F$-split.
See also
\cite{Kawakami-Tanaka2,Kawakami-Tanaka4} for further results on quasi-$F$-splitting of smooth Fano varieties.

\begin{prop}[Cubic hypersurfaces]\label{prop:cubic}
    Let $X$ be a smooth cubic 
    hypersurface of dimension $d \geq 2$.
Then one of the following holds:
\begin{enumerate}
    \item[\textup{(1)}] $\sht(X)=1$ or
    \item[\textup{(2)}] $\sht(X)=2$, $p=2$, and $X$ is isomorphic to the Fermat cubic.
\end{enumerate}
\end{prop}
\begin{proof}
For the proof, suppose that $X$ is not $F$-split.
We define $H_{d-1}$ as a smooth hyperplane section of $X$, and $H_{m-1}$ as a smooth hyperplane section of $H_{m}$ for $m\in\{d-1,\ldots 1\}$, inductively.
Then Fedder's criterion \cite{Fedder} shows that $H_2$ is never $F$-split, and thus it follows from \cite[Example 5.5]{Hara} that $p=2$ and $H_2$ is isomorphic to the Fermat cubic surface.
By Beauville’s theorem \cite[Theorem 2.5]{Cubic}, $H_3$ must be isomorphic to the Fermat cubic, and repeating this, we conclude that $X$ is isomorphic to the Fermat cubic.
Finally, by Proposition \ref{cor:hypersurface inversion of adjunction}, we get $\sht(X)\leq \sht(H_1)\leq 2$.
\end{proof}

\begin{prop}\label{prop:sm del Pezzo}
    Let $X$ be a smooth del Pezzo surface. Then $\sht(X)\leq 2$. 
\end{prop}
\begin{proof}
    By Corollary \ref{cor:base change}, after replacing the base field by its algebraic closure, we may assume that $X$ is defined over an algebraically closed field.
When $K_X^2>4$, \cite[Example 5.5]{Hara} shows that $\sht(X)=1$.
When $K_X^2\leq 3$, the del Pezzo surface $X$ is a hypersurface in a weighted projective space and a general member of the anti-canonical linear system is a smooth elliptic curve $C$ (see \cite[Theorem 2.15]{BT} and \cite[Theorem 1.4]{KN1}). 
Note that, since $X$ is a smooth hypersurface, it is contained in the smooth locus of the weighted projective space; in particular, $\sO_X(1)$ is Cartier.
Thus we have $-K_X=\sO_X(1)$ since $K_X^2\leq 3$.
Therefore, we can apply Proposition \ref{cor:hypersurface inversion of adjunction weighted} to obtain $\sht(X)\leq\sht(C)\leq 2$. 
\end{proof}

\begin{eg}[Non-quasi-$F$-split del Pezzo surface with RDPs]\label{eg:non-quasi-$F$-split Du Val del Pezzo}
Let $p:=2$ and $X$ a hypersurface in a weighted projective space $\mathbb{P}(1:1:1:2)_{[x:y:z:w]}$ defined by $\{f:=w^2+xyz(x+y+z)=0\}$.
Then $X$ is a del Pezzo surface with seven $A_1$-singularities (see \cite[Proposition 2.3]{KN2}).
Set $S:=k[x,y,z,w]$. Then $\sht(S/(f))=\infty$. 
In fact, we have
\begin{eqnarray*}
I_{\infty}(f) &=&(f,x^2y^2z+xy^2z^2,x^2y^2z+x^2yz^2,x^2y^2z+xw^2,x^2y^2z+yw^2, \\
&& x^2y^2z+zw^2, x^2yzw+xy^2zw,x^2yzw+xyz^2w).
\end{eqnarray*}
This can be verified using the Macaulay2 code available on the second author’s homepage (\cite{Takamatsu_code}).
We note that since the singular locus of $ \P (1:1:1:2)$ is $\{(0:0:0:1)\}$, $X$ is contained in the regular locus of the weighted projective space.
\end{eg}

\begin{eg}[Non-quasi-$F$-split smooth Fano hypersurfaces]\label{eg:non-quasi-F-split smooth Fano}
Let $X$ be a hypersurface in $\mathbb{P}^{p+1}$ defined by
$\{f:=x_0^{p+1}+\cdots+x_{p+1}^{p+1}=0\}$.
Then $X$ is a smooth Fano $p$-fold.
If $p \geq 3$, then $X$ is not quasi-$F$-split by Corollary \ref{cor:non-q-split criterion}.
Indeed, we have $f^{p-2} \in \m^{[p]}$ if $p \geq 3$.
In particular, taking the product with $\P^1$, we can deduce from Proposition \ref{prop:fiber space} that there is a non-quasi-$F$-split smooth Fano $d$-fold in characteristic three for all $d\geq 3$.
\end{eg}

\begin{eg}[Quasi-$F$-split wild conic bundle]\label{eg:wild conic bundle}
Let $p:=2$ and $X$ be a hypersurface in $\mathbb{P}_{[x_0:x_1:x_2]}^2\times \mathbb{P}_{[y_0:y_1:y_2]}^2$ defined by
$\{f:=x_0y_0^2+x_1y_1^2+x_2y_2^2=0\}$.
Then $X$ is a smooth Fano threefold, and the restriction of the first projection of $\mathbb{P}_{[x_0:x_1:x_2]}^2\times \mathbb{P}_{[y_0:y_1:y_2]}^2$ to $X$ gives a wild conic bundle structure, i.e.,\ all fibers are non-reduced.

On the other hand, it follows from Theorem  \ref{thm:Fedder's criterion for projective varieties in weighted case} that $\sht(X)=2$.
Indeed, for $g:=x_0y_1y_2f^2 \in I_1(f)$, we have
\[
u(F_*g)=0,\ \text{and}\ \theta_f(F_*g) \notin \m^{[p]}.
\]
This shows that the quasi-$F$-split property is not inherited by a general fiber.

Furthermore, we examine the behavior of quasi-$F$-splitting under field extensions, using the generic fiber of the above fibration.
Let $K=k(S,T)$ be the function field of the base scheme.
Then the generic fiber $X_K$ of the above fibration is a hypersurface in $\P^2_K$ defined by
$Sy_0^2+Ty_1^2+y_2^2$.
Since $X$ is quasi-$F$-split, so is $X_K$.
However, the base change to $\var{K}$ is non-reduced.
and in particular, is not quasi-$F$-split.
Therefore, Corollary \ref{cor:base change, Calabi-Yau variety} fails in the non-Calabi--Yau case.
Passing to section rings shows that \cite[Theorem~5.13]{KTY1} does not extend to the non-Calabi--Yau setting.
\end{eg}

\subsection{Rational double points}

In this subsection, we determine quasi-$F$-split heights of all RDPs.
We note that such computations have already been carried out for $D$-type in characteristic $p=2$ in unpublished work by Yobuko, but the proof is completely different.

\begin{thm}\label{rdp}
Every rational double point is quasi-$F$-split.
\end{thm}
We recall that RDPs in characteristic $p>5$ are $F$-regular (hence $F$-split). 
More generally, every taut RDP is $F$-regular by \cite[Theorem 1.1]{Hara2}.
Table \ref{table:RDPs} is a list of quasi-$F$-split heights of non-taut RDPs in characteristic $p=2,3,5$.

\begin{table}[ht]
\caption{Quasi-$F$-split heights of non-taut RDPs}
\centering
\begin{tabular}{|l|c|l|c|l|}
\hline
$p$ & type & \hspace{18mm} $f$ & ht($R/(f)$) & \\
\hline
\hline
2 & $D_{2n}^{0}$ &$z^2 +x^{2}y+xy^n$ & $\lceil \log_2 n \rceil +1 $&$n \geq 2$ \\
\hline
2 & $D_{2n}^{r}$ &$z^2 +x^{2}y+xy^n+xy^{n-r}z$ &$ \lceil \log_2 (n-r) \rceil +1 $&$r=1, \ldots, n-1$ \\
\hline
2 & $D_{2n+1}^{0}$ &$z^2+x^2y+y^nz$ & $\lceil \log_2 n \rceil +1$ & $n \geq 2$ \\
\hline
2 & $D_{2n+1}^{r}$ &$z^2+x^2y+y^nz+xy^{n-r}z$ &$\lceil \log_2 (n-r) \rceil +1 $& $r=1,\ldots, n-1$\\
\hline
2 & $E_{6}^{0}$ &$z^2+x^3+y^2z $& 2 & \\
\hline
2 & $E_{6}^{1}$ &$z^2+x^3+y^2z+xyz$ & 1 & \\
\hline
2 & $E_{7}^{0}$ &$z^2+x^3+xy^3$ & 4 & \\
\hline
2 & $E_{7}^{1}$ &$z^2+x^3+xy^3+x^2yz$& 3 & \\
\hline
2 & $E_{7}^{2}$ &$z^2+x^3+xy^3+y^3z$ & 2 & \\
\hline
2 & $E_{7}^{3}$ &$z^2+x^3+xy^3+xyz$& 1 & \\
\hline
2 & $E_{8}^{0}$ &$z^2+x^3+y^5$ & 4 & \\
\hline
2 & $E_{8}^{1}$ &$z^2+x^3+y^5+xy^3z$ & 4 & \\
\hline
2 & $E_{8}^{2}$ &$z^2+x^3+y^5+xy^2z$ & 3 & \\
\hline
2 & $E_{8}^{3}$ &$z^2+x^3+y^5+y^3z$ & 2 & \\
\hline
2 & $E_{8}^{4}$ &$z^2+x^3+y^5+xyz$& 1 & \\
\hline
3 & $E_{6}^{0}$ & $z^2+x^3+y^4$ &2 & \\ 
\hline
3 & $E_{6}^{1}$ &$z^2+x^3+y^4+x^2y^2$ &1 & \\ 
\hline
3 & $E_{7}^{0}$ & $z^2+x^3+xy^3$ &2 & \\ 
\hline
3 & $E_{7}^{1}$ &$z^2+x^3+xy^3+x^2y^2$ & 1& \\ 
\hline
3 & $E_{8}^{0}$ & $z^2+x^3+y^5$ & 3& \\ 
\hline
3 & $E_{8}^{1}$ & $z^2+x^3+y^5+x^2y^3$ &2 & \\ 
\hline
3 & $E_{8}^{2}$ & $z^2+x^3+y^5+x^2y^2$ &1 & \\ 
\hline
5 & $E_{8}^{0}$ &$z^2+x^3+y^5$  & 2& \\ 
\hline
5 & $E_{8}^{1}$ &$z^2+x^3+y^5+xy^4$ &1 & \\ 
\hline
\end{tabular}
\label{table:RDPs}
\end{table}

We explain the method of computation
for the $D_{2n}^{0}$-type in $p=2$.
We put 
\[
f:= z^2+x^2y+xy^n.
\]
Then we have
\[
\Delta_{1}(f) = x^2yz^2+xy^nz^2+x^3y^{n+1}.
\]
We put $m:=\lceil \log_2 n \rceil$.
First, we prove $\sht(R/(f)) \geq m+1$.
It is enough to show that $f\Delta_1(f)^{1+2+\cdots+2^{l-2}} \in \m^{[p^l]}$ for $l \leq m$.
Each term in $f\Delta_1(f)^{1+2+\cdots+2^{l-2}}$ has a form
\[
(z^2)^a(x^2y)^b(xy^n)^c=x^{2b+c}y^{b+nc}z^{2a}
\]
with $a+b+c=2^{l}-1$.
If it is not contained in $\m^{[p^l]}$, then we have
\[
2b+c \leq 2^l-1,\ b+nc \leq 2^l-1,\ 2a \leq 2^l-1.
\]
Therefore, we have $\log_2 n \leq l-1$, and thus $m+1 \leq l$.
Next, we prove $\sht(R/(f)) \leq m+1$.
First, we consider the case where $n$ is even.
We have the expansion
\[
F_*\Delta_{1}(f) = xzv_y+y^{n/2}zv_x+xy^{n/2}v_{xy}.
\]
We put $g_1:=zy^{2^m-n}f$. Then we have
\[
u(F_*g_1)=0,\ \text{and}\ \theta_f(F_*g_1)=xy^{2^{m-1}}z.
\]
We put $g_2:=xy^{2^{m-1}}z$. Then we have
\[
u(F_*g_2)=0,\ \text{and}\ \theta_f(F_*g_2)=xy^{2^{m-2}}z
\]
if $m > 1$.
Therefore, if we put $g_l:=xy^{2^{m+1-l}}z$ for $2 \leq l \leq m+1$, then we have $g_l \in I_l(f)$.
Furthermore, since $g_{m+1}=xyz \notin \m^{[p]}$, we have $\sht(R/(f)) \leq m +1$.

Next, we consider the case where $n$ is odd.
Then we put $g_1:=y^{2^m-n}zf$.
Since we have
\[
F_*\Delta_1(f)=xzv_y+y^{(n-1)/2}zv_{xy}+xy^{(n+1)/2}v_x,
\]
it follows that
\[
u(F_*g_1)=0,\ \text{and}\ \theta_f(F_*g_1)=xy^{2^{m-1}}z.
\]
If we put $g_2:=xy^{2^{m-1}}z$, then
\[
u(F_*g_2)=0,\ \text{and}\ \theta_f(F_*g_2)=xy^{2^{m-2}}z
\]
when $m >1$.
Therefore, if we put $g_l:=xy^{2^{m+1-l}}z$ for $2 \leq l \leq m+1$, then we have $g_l \in I_l(f)$.
Furthermore, since $g_{m+1}=xyz \notin \m^{[p]}$, we have $\sht(R/(f)) \leq m + 1$.

\printbibliography

@article{KTY1,
Author = {Kawakami, Tatsuro and Takamatsu, Teppei and Yoshikawa, Shou},
Title = {Fedder type criterion for quasi-{$F$}-splitting I},
Year = {2022},
Journal = {arXiv:2204.10076, to appear in Amer.~J.~Math.},
}

@Article{NT,
  author        = {{Nagamachi}, Ippei and {Takamatsu}, Teppei},
  journal       = {arXiv preprint arXiv:2110.03917},
  title         = {{On behavior of conductors, Picard schemes, and Jacobian numbers of varieties over imperfect fields}},
  year          = {2021}
}

@article {Ejiri,
    AUTHOR = {Ejiri, Sho},
     TITLE = {When is the {A}lbanese morphism an algebraic fiber space in
              positive characteristic?},
   JOURNAL = {Manuscripta Math.},
  FJOURNAL = {Manuscripta Mathematica},
    VOLUME = {160},
      YEAR = {2019},
    NUMBER = {1-2},
     PAGES = {239--264},
      ISSN = {0025-2611},
   MRCLASS = {14D06 (14G17)},
  MRNUMBER = {3983395},
MRREVIEWER = {Chen Jiang},
       DOI = {10.1007/s00229-018-1056-6},
       URL = {https://doi.org/10.1007/s00229-018-1056-6},
}

@article {LZ,
    AUTHOR = {Langer, Andreas and Zink, Thomas},
     TITLE = {De {R}ham-{W}itt cohomology for a proper and smooth morphism},
   JOURNAL = {J. Inst. Math. Jussieu},
  FJOURNAL = {Journal of the Institute of Mathematics of Jussieu. JIMJ.
              Journal de l'Institut de Math\'{e}matiques de Jussieu},
    VOLUME = {3},
      YEAR = {2004},
    NUMBER = {2},
     PAGES = {231--314},
      ISSN = {1474-7480},
   MRCLASS = {14F30 (14F40)},
  MRNUMBER = {2055710},
MRREVIEWER = {Martin C. Olsson},
       DOI = {10.1017/S1474748004000088},
       URL = {https://doi.org/10.1017/S1474748004000088},
}

@article {Yobuko19,
    AUTHOR = {Yobuko, Fuetaro},
     TITLE = {Quasi-{F}robenius splitting and lifting of {C}alabi-{Y}au
              varieties in characteristic {$p$}},
   JOURNAL = {Math. Z.},
  FJOURNAL = {Mathematische Zeitschrift},
    VOLUME = {292},
      YEAR = {2019},
    NUMBER = {1-2},
     PAGES = {307--316},
      ISSN = {0025-5874},
   MRCLASS = {14J32 (13F35 14G17)},
  MRNUMBER = {3968903},
MRREVIEWER = {Tyler L. Kelly},
       DOI = {10.1007/s00209-018-2198-7},
       URL = {https://doi.org/10.1007/s00209-018-2198-7},
}

@article {KN1,
    AUTHOR = {Kawakami, Tatsuro and Nagaoka, Masaru},
     TITLE = {Pathologies and liftability of {D}u {V}al del {P}ezzo surfaces
              in positive characteristic},
   JOURNAL = {Math. Z.},
  FJOURNAL = {Mathematische Zeitschrift},
    VOLUME = {301},
      YEAR = {2022},
    NUMBER = {3},
     PAGES = {2975--3017},
      ISSN = {0025-5874},
   MRCLASS = {14},
   MRNUMBER = {4437346},
       DOI = {10.1007/s00209-022-02998-6},
       URL = {https://doi.org/10.1007/s00209-022-02998-6},
}

@article {Hara,
    AUTHOR = {Hara, Nobuo},
     TITLE = {A characterization of rational singularities in terms of
              injectivity of {F}robenius maps},
   JOURNAL = {Amer. J. Math.},
  FJOURNAL = {American Journal of Mathematics},
    VOLUME = {120},
      YEAR = {1998},
    NUMBER = {5},
     PAGES = {981--996},
      ISSN = {0002-9327},
   MRCLASS = {13A99 (14B05 14E05)},
  MRNUMBER = {1646049},
MRREVIEWER = {Karen E. Smith},
       URL =
              {http://muse.jhu.edu/journals/american_journal_of_mathematics/v120/120.5hara.pdf},
}

@article {Hara2,
    AUTHOR = {Hara, Nobuo},
     TITLE = {Classification of two-dimensional {$F$}-regular and {$F$}-pure
              singularities},
   JOURNAL = {Adv. Math.},
  FJOURNAL = {Advances in Mathematics},
    VOLUME = {133},
      YEAR = {1998},
    NUMBER = {1},
     PAGES = {33--53},
      ISSN = {0001-8708},
   MRCLASS = {14J17 (14B05)},
  MRNUMBER = {1492785},
MRREVIEWER = {Karen E. Smith},
       DOI = {10.1006/aima.1997.1682},
       URL = {https://doi.org/10.1006/aima.1997.1682},
}

@incollection {Yobuko2,
    AUTHOR = {Yobuko, Fuetaro},
     TITLE = {On the {F}robenius-splitting height of varieties in positive
              characteristic},
 BOOKTITLE = {Algebraic number theory and related topics 2016},
    SERIES = {RIMS K\^{o}ky\^{u}roku Bessatsu, B77},
     PAGES = {159--175},
 PUBLISHER = {Res. Inst. Math. Sci. (RIMS), Kyoto},
      YEAR = {2020},
   MRCLASS = {14G17 (14J28)},
  MRNUMBER = {4278982},
}

@article {Fedder,
    AUTHOR = {Fedder, Richard},
     TITLE = {{$F$}-purity and rational singularity},
   JOURNAL = {Trans. Amer. Math. Soc.},
  FJOURNAL = {Transactions of the American Mathematical Society},
    VOLUME = {278},
      YEAR = {1983},
    NUMBER = {2},
     PAGES = {461--480},
      ISSN = {0002-9947},
   MRCLASS = {13H10 (13D03 14B05)},
  MRNUMBER = {701505},
MRREVIEWER = {D. Kirby},
       DOI = {10.2307/1999165},
       URL = {https://doi-org.utokyo.idm.oclc.org/10.2307/1999165},
}

@incollection {Bombieri2,
    AUTHOR = {Bombieri, E. and Mumford, D.},
     TITLE = {Enriques' classification of surfaces in char. {$p$}. {II}},
 BOOKTITLE = {Complex analysis and algebraic geometry},
     PAGES = {23--42},
      YEAR = {1977},
   MRCLASS = {14J10},
  MRNUMBER = {0491719},
MRREVIEWER = {S. L. Kleiman},
}

@article {tanaka22,
    AUTHOR = {Tanaka, Hiromu},
     TITLE = {Vanishing theorems of {K}odaira type for {W}itt canonical
              sheaves},
   JOURNAL = {Selecta Math. (N.S.)},
  FJOURNAL = {Selecta Mathematica. New Series},
    VOLUME = {28},
      YEAR = {2022},
    NUMBER = {1},
     PAGES = {Paper No. 12, 50},
      ISSN = {1022-1824},
   MRCLASS = {14F30 (14F17)},
  MRNUMBER = {4346509},
       DOI = {10.1007/s00029-021-00736-0},
       URL = {https://doi-org.utokyo.idm.oclc.org/10.1007/s00029-021-00736-0},
}

@article {KN2,
    AUTHOR = {Kawakami, Tatsuro and Nagaoka, Masaru},
     TITLE = {Classification of {D}u {V}al del {P}ezzo surfaces of {P}icard
              rank one in characteristic two and three},
   JOURNAL = {J. Algebra},
  FJOURNAL = {Journal of Algebra},
    VOLUME = {636},
      YEAR = {2023},
     PAGES = {603--625},
      ISSN = {0021-8693,1090-266X},
   MRCLASS = {14J26 (13A35 14G17 14J45)},
  MRNUMBER = {4644313},
       DOI = {10.1016/j.jalgebra.2023.08.027},
       URL = {https://doi.org/10.1016/j.jalgebra.2023.08.027},
}

@incollection {Dol,
    AUTHOR = {Dolgachev, Igor},
     TITLE = {Weighted projective varieties},
 BOOKTITLE = {Group actions and vector fields ({V}ancouver, {B}.{C}., 1981)},
    SERIES = {Lecture Notes in Math.},
    VOLUME = {956},
     PAGES = {34--71},
 PUBLISHER = {Springer, Berlin},
      YEAR = {1982},
   MRCLASS = {14L32 (14A05 14B05)},
  MRNUMBER = {704986},
       DOI = {10.1007/BFb0101508},
       URL = {https://doi-org.utokyo.idm.oclc.org/10.1007/BFb0101508},
}

@article {GK2,
    AUTHOR = {van der Geer, G. and Katsura, T.},
     TITLE = {On the height of {C}alabi-{Y}au varieties in positive
              characteristic},
   JOURNAL = {Doc. Math.},
  FJOURNAL = {Documenta Mathematica},
    VOLUME = {8},
      YEAR = {2003},
     PAGES = {97--113},
      ISSN = {1431-0635},
   MRCLASS = {14J32 (14L05)},
  MRNUMBER = {2029163},
MRREVIEWER = {Mark Gross},
}

@article {KTTWYY,
    AUTHOR = {Kawakami, Tatsuro and Takamatsu, Teppei and Tanaka, Hiromu and
              Witaszek, Jakub and Yobuko, Fuetaro and Yoshikawa, Shou},
     TITLE = {Quasi-{$F$}-splittings in birational geometry},
   JOURNAL = {Ann. Sci. \'Ec. Norm. Sup\'er. (4)},
  FJOURNAL = {Annales Scientifiques de l'\'Ecole Normale Sup\'erieure.
              Quatri\`eme S\'erie},
    VOLUME = {58},
      YEAR = {2025},
    NUMBER = {3},
     PAGES = {665--748},
      ISSN = {0012-9593,1873-2151},
   MRCLASS = {14E05 (13A35 14G17)},
  MRNUMBER = {4962159},
}

@article {Kawakami-Tanaka2,
    AUTHOR = {Kawakami, Tatsuro and Tanaka, Hiromu},
     TITLE = {Weak quasi-{$F$}-splitting and del {P}ezzo varieties},
   JOURNAL = {J. Lond. Math. Soc. (2)},
  FJOURNAL = {Journal of the London Mathematical Society. Second Series},
    VOLUME = {111},
      YEAR = {2025},
    NUMBER = {2},
     PAGES = {Paper No. e70098, 47},
      ISSN = {0024-6107,1469-7750},
   MRCLASS = {14J45 (13A35 14G17 14J26)},
  MRNUMBER = {4868758},
       DOI = {10.1112/jlms.70098},
       URL = {https://doi-org.utokyo.idm.oclc.org/10.1112/jlms.70098},
}

@article {KTTWYY2,
    AUTHOR = {Kawakami, Tatsuro and Takamatsu, Teppei and Tanaka, Hiromu and
              Witaszek, Jakub and Yobuko, Fuetaro and Yoshikawa, Shou},
     TITLE = {Quasi-{$F$}-splittings in birational geometry {II}},
   JOURNAL = {Proc. Lond. Math. Soc. (3)},
  FJOURNAL = {Proceedings of the London Mathematical Society. Third Series},
    VOLUME = {128},
      YEAR = {2024},
    NUMBER = {4},
     PAGES = {Paper No. e12593, 81},
      ISSN = {0024-6115,1460-244X},
   MRCLASS = {14E30 (13A35 14G17)},
  MRNUMBER = {4731853},
MRREVIEWER = {Karl\ Schwede},
       DOI = {10.1112/plms.12593},
       URL = {https://doi-org.utokyo.idm.oclc.org/10.1112/plms.12593},
}

@article{KTTWYY3,
Author = {Kawakami, Tatsuro and Takamatsu, Teppei and Tanaka, Hiromu and
              Witaszek, Jakub and Yobuko, Fuetaro and Yoshikawa, Shou},
Title = {Quasi-{$F$}-splittings in birational geometry {III}},
Year = {2024},
Journal = {arXiv:2408.01921},
}

@article{Kawakami-Tanaka4,
Author = {Kawakami, Tatsuro and Tanaka, Hiromu},
Title = {Liftability and vanishing theorems for Fano threefolds in positive characteristic II},
Year = {2024},
Journal = {arXiv:2404.04764},
}

@article{YobukoHodgeWitt,
Author = {Yobuko, Fuetaro},
Title = {Quasi-F-split and Hodge-Witt},
Year = {2023},
Journal = {arXiv:2312.00682},
}

@article{TWY,
Author = {Tanaka, Hiromu and Witaszek, Jakub and Yobuko, Fuetaro},
Title = {Quasi-{$F$}-regularity and quasi-$+$-regularity},
Year = {2023},
Journal = {In preparation},
}

@article {Cubic,
    AUTHOR = {Kadyrsizova, Zhibek and Kenkel, Jennifer and Page, Janet and
              Singh, Jyoti and Smith, Karen E. and Vraciu, Adela and Witt,
              Emily E.},
     TITLE = {Cubic surfaces of characteristic two},
   JOURNAL = {Trans. Amer. Math. Soc.},
  FJOURNAL = {Transactions of the American Mathematical Society},
    VOLUME = {374},
      YEAR = {2021},
    NUMBER = {9},
     PAGES = {6251--6267},
      ISSN = {0002-9947},
   MRCLASS = {14G17 (13A35 14J26)},
  MRNUMBER = {4302160},
       DOI = {10.1090/tran/8341},
       URL = {https://doi-org.utokyo.idm.oclc.org/10.1090/tran/8341},
}

@article {AKMMMP,
    AUTHOR = {An, Sang Yook and Kim, Seog Young and Marshall, David C. and
              Marshall, Susan H. and McCallum, William G. and Perlis,
              Alexander R.},
     TITLE = {Jacobians of genus one curves},
   JOURNAL = {J. Number Theory},
  FJOURNAL = {Journal of Number Theory},
    VOLUME = {90},
      YEAR = {2001},
    NUMBER = {2},
     PAGES = {304--315},
      ISSN = {0022-314X},
   MRCLASS = {14H40 (11G05 14C20 14C40 14H52)},
  MRNUMBER = {1858080},
MRREVIEWER = {V. V. Chueshev},
       DOI = {10.1006/jnth.2000.2632},
       URL = {https://doi-org.utokyo.idm.oclc.org/10.1006/jnth.2000.2632},
}

@article{PW,
  title={Singularities of general fibers and the LMMP},
  author={Patakfalvi, Zsolt and Waldron, Joe},
  journal={American Journal of Mathematics},
  volume={144},
  number={2},
  pages={505--540},
  year={2022},
  publisher={Johns Hopkins University Press}
}

@incollection {Artin,
    AUTHOR = {Artin, M.},
     TITLE = {Coverings of the rational double points in characteristic
              {$p$}},
 BOOKTITLE = {Complex analysis and algebraic geometry},
     PAGES = {11--22},
      YEAR = {1977},
   MRCLASS = {14B05},
  MRNUMBER = {0450263},
MRREVIEWER = {Jonathan M. Wahl},
}

@article {tanaka21,
    AUTHOR = {Tanaka, Hiromu},
     TITLE = {Invariants of algebraic varieties over imperfect fields},
   JOURNAL = {Tohoku Math. J. (2)},
  FJOURNAL = {The Tohoku Mathematical Journal. Second Series},
    VOLUME = {73},
      YEAR = {2021},
    NUMBER = {4},
     PAGES = {471--538},
      ISSN = {0040-8735},
   MRCLASS = {14G17 (14D06)},
  MRNUMBER = {4355058},
       DOI = {10.2748/tmj.20200611},
       URL = {https://doi.org/10.2748/tmj.20200611},
}

@article {BT,
    AUTHOR = {Bernasconi, Fabio and Tanaka, Hiromu},
     TITLE = {On del {P}ezzo fibrations in positive characteristic},
   JOURNAL = {J. Inst. Math. Jussieu},
  FJOURNAL = {Journal of the Institute of Mathematics of Jussieu. JIMJ.
              Journal de l'Institut de Math\'{e}matiques de Jussieu},
    VOLUME = {21},
      YEAR = {2022},
    NUMBER = {1},
     PAGES = {197--239},
      ISSN = {1474-7480},
   MRCLASS = {14E30 (14G17 14J45)},
  MRNUMBER = {4366337},
       DOI = {10.1017/S1474748020000067},
       URL = {https://doi.org/10.1017/S1474748020000067},
}

@misc{Takamatsu_code,
  author       = {Teppei Takamatsu},
  title        = {Macaulay 2 scripts for Fedder type criteria for quasi-F-splitting},
  howpublished = {available at \url{https://sites.google.com/view/teppei-takamatsu/home/scripts}},
}

@book {Silverman,
    AUTHOR = {Silverman, Joseph H.},
     TITLE = {The arithmetic of elliptic curves},
    SERIES = {Graduate Texts in Mathematics},
    VOLUME = {106},
 PUBLISHER = {Springer-Verlag, New York},
      YEAR = {1986},
     PAGES = {xii+400},
      ISBN = {0-387-96203-4},
   MRCLASS = {11G05 (14Gxx 14K07 14K15)},
  MRNUMBER = {817210},
MRREVIEWER = {Robert\ S.\ Rumely},
       DOI = {10.1007/978-1-4757-1920-8},
       URL = {https://doi-org.kyoto-u.idm.oclc.org/10.1007/978-1-4757-1920-8},
}

@book {Mumford,
    AUTHOR = {Mumford, David},
     TITLE = {Abelian varieties},
    SERIES = {Tata Institute of Fundamental Research Studies in Mathematics},
    VOLUME = {5},
 PUBLISHER = {Tata Institute of Fundamental Research, Bombay; by Oxford
              University Press, London},
      YEAR = {1970},
     PAGES = {viii+242},
   MRCLASS = {14.51},
  MRNUMBER = {282985},
MRREVIEWER = {James\ Milne},
}

@article {Mehta-Srinivas,
    AUTHOR = {Mehta, V. B. and Srinivas, V.},
     TITLE = {Varieties in positive characteristic with trivial tangent
              bundle},
      NOTE = {With an appendix by Srinivas and M. V. Nori},
   JOURNAL = {Compositio Math.},
  FJOURNAL = {Compositio Mathematica},
    VOLUME = {64},
      YEAR = {1987},
    NUMBER = {2},
     PAGES = {191--212},
      ISSN = {0010-437X,1570-5846},
   MRCLASS = {14E20 (14K99 14L30)},
  MRNUMBER = {916481},
MRREVIEWER = {H.\ Lange},
       URL = {http://www.numdam.org/item?id=CM_1987__64_2_191_0},
}
\end{document}